\documentclass[reqno]{amsart}

\usepackage{amsmath, amssymb, amsfonts}
\usepackage{amsthm}
\usepackage{mathrsfs}
\usepackage{mathtools}
\usepackage[hidelinks]{hyperref}

\newcommand{\Z}{\mathbb{Z}} 
 
\newcommand{\R}{\mathbb{R}} 
\newcommand{\C}{\mathbb{C}}
\renewcommand{\H}{\mathbb{H}}

\newcommand{\CP}{\mathbb{CP}}

\newcommand{\A}{\mathcal{A}}

\newcommand{\G}{\mathcal{G}}
\newcommand{\cK}{\mathcal{K}}
\renewcommand{\L}{\mathcal{L}}
\newcommand{\M}{\mathcal{M}}
\newcommand{\N}{\mathcal{N}}
\renewcommand{\O}{\mathcal{O}}

\newcommand{\g}{\mathfrak{g}}
\newcommand{\h}{\mathfrak{h}}
\renewcommand{\k}{\mathfrak{k}}
\renewcommand{\l}{\mathfrak{l}}

\renewcommand{\t}{\mathfrak{t}}
\newcommand{\gl}{\mathfrak{gl}}
\renewcommand{\sl}{\mathfrak{sl}}

\newcommand{\I}{\mathsf{I}}
\newcommand{\J}{\mathsf{J}}
\newcommand{\K}{\mathsf{K}}

\newcommand{\LD}{\mathscr{L}}

\newcommand{\s}{\subseteq}

\renewcommand{\d}{\partial}
\newcommand{\db}{\bar{\d}}
\newcommand{\ip}[1]{\langle #1\rangle}
\newcommand{\e}{\varepsilon}

\newcommand{\too}{\longrightarrow}
\newcommand{\mto}{\mapsto}
\newcommand{\mtoo}{\longmapsto}

\DeclareMathOperator{\GL}{GL}
\DeclareMathOperator{\SL}{SL}

\DeclareMathOperator{\U}{U}
\DeclareMathOperator{\SU}{SU}
\DeclareMathOperator{\SO}{SO}

\DeclareMathOperator{\Lie}{Lie}

\DeclareMathOperator{\Ad}{Ad}

\DeclareMathOperator{\Spec}{Spec}
\DeclareMathOperator{\Proj}{Proj}

\let\Im\relax\DeclareMathOperator{\Im}{Im}

\newcommand{\ass}[1]{^{#1{\operatorname{-ss}}}}
\newcommand{\aps}[1]{^{#1{\operatorname{-ps}}}}

\DeclareFontFamily{U}{MnSymbolC}{}
\DeclareSymbolFont{MnSyC}{U}{MnSymbolC}{m}{n}
\DeclareFontShape{U}{MnSymbolC}{m}{n}{
    <-6>  MnSymbolC5
   <6-7>  MnSymbolC6
   <7-8>  MnSymbolC7
   <8-9>  MnSymbolC8
   <9-10> MnSymbolC9
  <10-12> MnSymbolC10
  <12->   MnSymbolC12}{}
\DeclareMathSymbol{\intp}{\mathbin}{MnSyC}{'270}

\newcommand{\sll}[1]{\mkern-4mu\mathbin{/\mkern-5mu/}_{\mkern-4mu{#1}}}

\newcommand{\ddt}[1]{\frac{d}{dt}\Big|_{t={#1}}}

\theoremstyle{plain}
\newtheorem{theorem}{Theorem}[section]
\newtheorem{lemma}[theorem]{Lemma}
\newtheorem{proposition}[theorem]{Proposition}

\theoremstyle{definition}
\newtheorem{definition}[theorem]{Definition}
\newtheorem{example}[theorem]{Example}
\newtheorem{remark}[theorem]{Remark}

\numberwithin{figure}{section}
\numberwithin{equation}{section}

\newtheorem*{step*}{Step}

\newtheorem*{claim*}{Claim}

\title{Kempf--Ness type theorems and Nahm equations}
\author{Maxence Mayrand}
\thanks{During the preparation of this paper, the author was supported by a Moussouris Scholarship from the University of Oxford and a PGS D scholarship from the Natural Sciences and Engineering Research Council of Canada (NSERC)}
\address{Maxence Mayrand\\ Mathematical Institute, Andrew Wiles Building\\ University of Oxford\\ Oxford, OX2 6GG\\ United Kingdom}
\email{maxence.mayrand@maths.ox.ac.uk}

\newcommand{\introcpt}{K}
\newcommand{\introred}{G}
\newcommand{\introcptla}{\mathfrak{\k}}

\newcommand{\introredel}{g}
\newcommand{\aoran}{a\ } 

\begin{document}

\begin{abstract}
We prove a version of the affine Kempf--Ness theorem for non-algebraic symplectic structures and shifted moment maps, and use it to describe hyperk\"ahler quotients of $T^*G$, where $G$ is a complex reductive group.
\end{abstract}

\maketitle

\section{Introduction}

\subsection{Overview}\label{1trw7lra}

Broadly speaking, Kempf--Ness type theorems (named after \cite{kem79}) identify certain symplectic reductions \cite{mar74} with geometric invariant theory (GIT) quotients \cite{mum94}, thus providing an interesting bridge between differential and algebraic geometry. In the simplest form, it says that if $M\s\C^n$ is a smooth complex affine variety endowed with the standard symplectic structure and $\introcpt$ is a closed subgroup of the unitary group $\U(n)$ preserving $M$, then there is a moment map $\mu$ for the action of $\introcpt$ on $M$ such that the symplectic reduction $\mu^{-1}(0)/\introcpt$ is homeomorphic to the GIT quotient $M\sll{}\introred\coloneqq\Spec\C[M]^{\introred}$, where $\introred\coloneqq \introcpt_\C\s\GL(n,\C)$ is the complexification of $\introcpt$ (see e.g.\ \cite[Corollary 4.7]{sch89}). The moment map $\mu$ has the explicit expression
\begin{equation}\label{g32kdsjq}
\mu:M\too\introcptla^*,\quad\mu(p)(X) = -\frac{1}{2}\Im\ip{Xp,p},\quad(p\in M,X\in\introcptla)
\end{equation}
where $\introcptla\coloneqq\Lie(\introcpt)$ and $\ip{\cdot,\cdot}$ is the standard inner-product on $\C^n$. Recall that $G$ is reductive so $M\sll{}G$ is an affine variety. Moreover, if $\mu^{-1}(0)/\introcpt$ is smooth, then its reduced symplectic form is a K\"ahler form on $M\sll{}\introred$.

This theorem admits many generalizations and variants; for instance, there are versions for projective manifolds \cite[\S2]{nes84} \cite[\S8]{kir84} \cite[\S8.2]{mum94} \cite[\S4]{gui82-geo} \cite[\S2.2]{sja95}. Another important version -- which is closer to the spirit of this paper -- is when $M$ is affine as above but we shift the moment map \eqref{g32kdsjq}. More precisely, if $\chi:\introcpt\to S^1$ is a character, then $\xi\coloneqq i\,d\chi\in\introcptla^*$ is central (i.e.\ fixed by the coadjoint action), so we can consider the symplectic reduction $\mu^{-1}(\xi)/\introcpt$. Then, King \cite[\S6]{kin94} (see also \cite{hos14}) showed that $\mu^{-1}(\xi)/\introcpt$ is homeomorphic to the twisted GIT quotient $M\sll{\chi}\introred$, i.e.\ the GIT quotient of $M$ by $\introred$ with respect to the trivial line bundle $M\times\C$ with the $\introred$-action $\introredel\cdot(p,z)=(\introredel\cdot p,\chi(\introredel)z)$. In other words,
$$M\sll{\chi}\introred\coloneqq\Proj\bigoplus_{n=0}^\infty\C[M]^{\introred,\chi^n}$$
where $\C[M]^{\introred,\chi^n}$ is the set of $u\in\C[M]$ such that $u(\introredel\cdot p)=\chi(\introredel)^nu(p)$ for all $p\in M$ and $\introredel\in{\introred}$. Recall that the quasi-projective variety $M\sll{\chi}\introred$ is a categorical quotient for the action of $\introred$ on the set of semistable points
$$
M\ass{\chi}\coloneqq\{p\in M:\exists n\geq1\text{ and }u\in\C[M]^{\introred,\chi^n}\text{ such that }u(p)\ne 0\}.
$$
Again, if the quotient is smooth, the reduced symplectic form is K\"ahler.

There is another useful generalization, which is to consider symplectic reduction with respect to a symplectic form on $M\s\C^n$ which is not necessarily the standard one. More precisely, take a K\"ahler potential on $M$, i.e.\ a smooth function $f:M\to\R$ such that the 2-form $\omega\coloneqq 2i\d\db f$ is symplectic. Then, if $f$ is $\introcpt$-invariant, proper, and bounded below, there is still a moment map $\mu$ (not necessarily the same as above) such that $\mu^{-1}(0)/\introcpt\cong M\sll{}\introred$ (see e.g. \cite[Lemma 6.1]{hei00} or \cite[Proposition 4.2]{may17}). The standard version can be recovered by taking $f=\frac{1}{2}\|\cdot\|^2$.

It is then natural to try to combine these two versions, i.e.\ to shift the moment map $\mu$ associated to a K\"ahler potential $f$ on $M$ by $\xi\coloneqq i\,d\chi$, where $\chi:\introcpt\to S^1$ is a character. However, this requires more care into the relationship between $f$ and the algebraic structure of $M$. In general, $\mu^{-1}(\xi)/\introcpt$ can fail to be homeomorphic to $M\sll{\chi}\introred$, even if $f$ is $\introcpt$-invariant, proper, and bounded below (see \S\ref{xfc6kidf}). The first goal of this paper is to give a sufficient condition for this homeomorphism to hold. More precisely, we require that $e^f$ dominates polynomials on $M$, denoted $\C[M]\s o(e^f)$ (see Definition \ref{18hr8xku}). Informally, this means that
$$\lim_{p\to\infty}\frac{u(p)}{e^{f(p)}}=0,\quad\text{for all }u\in\C[M].$$
This holds, for example, in the standard case since $f=\frac{1}{2}\|\cdot\|^2$. We will later give a non-trivial example coming from the so-called \emph{Nahm equations}. But first, let us state our Kempf--Ness type theorem more precisely:

\begin{theorem}\label{77svbtma}
Let $M$ be a smooth complex affine variety, $\introcpt$ a compact Lie group acting on $M$ such that the action map $\introcpt\times M\to M$ is real algebraic, $f:M\to \R$ \aoran $\introcpt$-invariant K\"ahler potential such that $\C[M]\s o(e^f)$, and $\chi:\introcpt\to S^1$ a character. Then:
\begin{itemize}
\item[\textup{(1)}] The action of $\introcpt$ on $M$ extends to a complex algebraic action of the complexification $\introred\coloneqq \introcpt_\C$.
\item[\textup{(2)}] Let $\k\coloneqq\Lie(K)$ and let $\I$ be the complex structure on $M$. Then, the map
$$\mu:M\too\introcptla^*,\quad \mu(p)(X)=df(\I X^\#_p),\quad(p\in M,X\in\introcptla)$$
where $X^\#$ is the vector field generated by $X\in\introcptla$, is a moment map for the action of $\introcpt$ on $M$ with respect to the symplectic form $\omega\coloneqq 2i\d\db f$.
\item[\textup{(3)}] Let $\xi\coloneqq ai\,d\chi\in\introcptla^*$ for $a>0$. Then,
\begin{equation}\label{isil4gka}
M\ass{\chi}=\{p\in M:\overline{\introred\cdot p}\cap\mu^{-1}(\xi)\ne\emptyset\}.
\end{equation}
In particular, $\mu^{-1}(\xi)\s M\ass{\chi}$ and this inclusion descends to a homeomorphism $\mu^{-1}(\xi)/\introcpt\cong M\sll{\chi}\introred$.
\item[(4)] The action of $\introcpt$ on $\mu^{-1}(\xi)$ is free if and only if the action of $\introred$ on $M\ass{\chi}$ is free. In that case,  $\mu^{-1}(\xi)/\introcpt$ and $M\sll{\chi}\introred$ are smooth, and the reduced symplectic form on $\mu^{-1}(\xi)/\introcpt$ is a K\"ahler form with respect to the complex structure on $M\sll{\chi}\introred$. 
\end{itemize}
\end{theorem}

Parts (1) and (2) are well-known (and do not require $\C[M]\s o(e^f)$), and (4) follows easily from (3). The main step is to show \eqref{isil4gka}, which is largely inspired by King's proof in \cite{kin94}. We deduce the homeomorphism $\mu^{-1}(\xi)/\introcpt\cong M\sll{\chi}\introred$ from \eqref{isil4gka} by general results of Heinzner--Loose \cite{hei94}. If $\chi=1$, the condition $\C[M]\s o(e^f)$ can be replaced by the weaker condition that $f$ is proper and bounded below (see e.g.\ \cite[Lemma 6.1]{hei00} or \cite[Proposition 4.2]{may17}). It is easy to see that if $\C[M]\s o(e^f)$ then $f$ is proper and bounded below.

Let $K$ be a compact connected Lie group and let $G\coloneqq K_\C$. The second goal of this paper is to show that the conditions of the above theorem are satisfied for an interesting K\"ahler potential on the cotangent bundle $T^*G$ which is invariant under the left and right actions of $K$. This potential is defined in terms of the Nahm equations (reviewed below) and the corresponding Riemannian metric $g$, first discovered by Kronheimer \cite{kro88}, has the remarkable property of being hyperk\"ahler, i.e.\ there are three complex structures $\I, \J, \K$ on $T^*G$ (where $\I$ is the natural one) that are K\"ahler with respect to $g$ and satisfy $\I\J=\K$. Our motivation for studying this potential is to obtain explicit descriptions of the hyperk\"ahler quotients of $T^*G$ by closed subgroups of $K\times K$ as quasi-projective algebraic varieties; this will be explained in more details in \S\ref{m0mywa60} just below. 

Let us first explain a corollary of these results. The action of $G\times G$ on $G$ given by $(a,b)\cdot g=agb^{-1}$ lifts to an action on $T^*G$. Let $\g\coloneqq\Lie(G)$. By identifying $T^*G$ with $G\times\g^*$ using right translations, this action is $(a,b)\cdot(g,\xi)=(agb^{-1},\Ad_a^*\xi)$. Moreover, it preserves the canonical complex-symplectic form and there a moment map, namely 
\begin{equation}\label{2jztwdve}
\Phi:T^*G=G\times\g^*\too\g^*\times\g^*,\quad(g,\xi)\mtoo(\xi,-\Ad_{g^{-1}}^*\xi).
\end{equation}

\begin{theorem}\label{afdiktv2}
Let $G$ be a connected complex reductive group and let $H\s G\times G$ be a reductive subgroup. Let $\eta\in\h^*\coloneqq\Lie(H)^*$ be central and let $\chi:H\to\C^*$ be a character. Consider the moment map $\Phi_\h\coloneqq\Phi|_\h$ for the action of $H$ on $T^*G$. If $H$ acts freely on $\Phi^{-1}_\h(\eta)\ass{\chi}$, then the \textup{GIT} quotient $\Phi^{-1}_\h(\eta)\sll{\chi}H$ admits a complete hyperk\"ahler metric compatible with the complex-symplectic structure of $T^*G$.
\end{theorem}

Here, a {\it complete} hyperk\"ahler structure is one whose Riemannian metric is geodes\-ically complete. The compatibility condition means that if $(g,\I,\J,\K)$ is the hyperk\"ahler structure and $\omega_\I,\omega_\J,\omega_\K$ are the associated K\"ahler forms, then $\I$ is the natural complex structure on $\Phi^{-1}_\h(\eta)\sll{\chi}H$ and the pullback of the $\I$-complex-symplectic form $\omega_\J+i\omega_\K$ to $\Phi^{-1}_\h(\eta)\ass{\chi}$ is the restriction of the canonical complex-symplectic form on $T^*G$.

\subsection{Hyperk\"ahler quotients of $T^*G$}\label{m0mywa60}

We now give more details on our result on the K\"ahler potential of $T^*G$ and its implication for hyperk\"ahler quotients of $T^*G$. Throughout this section, $K$ is a compact connected Lie group and $G\coloneqq K_\C$ (equivalently, $G$ is a connected complex reductive group and $K\s G$ a maximal compact subgroup).

Recall that when a compact Lie group $L$ acts on a hyperk\"ahler manifold\break $(M,g,\I,\J,\K)$, a \textit{hyperk\"ahler moment map} is a map $\mu=(\mu_\I,\mu_\J,\mu_\K):M\to\l^*\times\l^*\times\l^*$ such that $\mu_\I$ is a moment map with respect to $\omega_\I$ (the K\"ahler forms of $(g,\I)$) and similarly for $\mu_\J$ and $\mu_\K$. This notion was introduced in \cite[\S3]{hit87}, where it is shown that if $L$ acts freely on $\mu^{-1}(\xi)$ for some central $\xi\in\l^*\times\l^*\times\l^*$, then the quotient $\mu^{-1}(\xi)/L$ is a smooth hyperk\"ahler manifold called the \textit{hyperk\"ahler quotient} of $M$ by $L$ at level $\xi$. Moreover, if the metric $g$ of $M$ is complete, then so is the metric of $\mu^{-1}(\xi)/L$. We will show that the hyperk\"ahler quotients of $T^*G$ by closed subgroups $L$ of $K\times K$ can be identified with the quasi-projective varieties $\Phi^{-1}(\xi)\sll{\chi}H$ in Theorem \ref{afdiktv2}.

Let us first recall, following Kronheimer \cite{kro88}, how to endow $T^*G$ with a hyperk\"ahler structure (see also \cite{dan96,bie97,tak15}). This uses the Nahm equations, which are a non-linear system of ordinary differential equations arising naturally from gauge theory as a dimensional reduction of the anti-self-dual Yang-Mills equations. Concretely, the Nahm equations are
\begin{equation*}
\begin{aligned}
\dot{A}_1 + [A_0, A_1] + [A_2, A_3] &=0 \\
\dot{A}_2 + [A_0, A_2] + [A_3, A_1] &=0 \\
\dot{A}_3 + [A_0, A_3] + [A_1, A_2] &=0
\end{aligned}
\end{equation*}
where $A_i:I\to\k\coloneqq\Lie(K)$ for $i=0,1,2,3$ and some interval $I\s\R$. Let $\A$ be the set of $C^1$ solutions to the Nahm equations on the interval $I=[0,1]$. There is a natural\footnote{By viewing $A$ as a connection on the trivial principal $K$-bundle over $[0,1]\times\R^3$, this is a gauge transformation, i.e.\ the pullback of a bundle automorphism.} action of the group $\cK$ of $C^2$ maps $[0,1]\to K$ on $\A$ given by
$$k\cdot(A_0,A_1,A_2,A_3)=(kA_0k^{-1}-\dot{k}k^{-1},kA_1k^{-1},kA_2k^{-1},kA_3k^{-1}).$$
Let $\cK^0$ be the subgroup of all $k\in\cK$ such that $k(0)=k(1)=1$. Then, the \textit{moduli space of solutions to the Nahm equations on $[0,1]$} is the quotient space
$$\M\coloneqq \A/\cK^0.$$
Note that there is a residual action of the group $\cK/\cK^0=K\times K$ on $\M$.

\begin{theorem}[Kronheimer \cite{kro88}]\label{knvp0kum}
The space $\M$ is a finite-dimensional smooth manifold and, for each choice of $K$-invariant inner-product on $\k$, there is a complete hyperk\"ahler structure $(g,\I,\J,\K)$ on $\M$ invariant under $K\times K$. Moreover, there is an isomorphism
$$\varphi:\M\too T^*G$$
of complex-symplectic manifolds \textup{(}where $\M$ has the complex-symplectic structure $(\I,\omega_\J+i\omega_\K)$ and $T^*G$ the canonical one\textup{)} which intertwines the two actions of $K\times K$.\qed
\end{theorem}

The hyperk\"ahler structure on $\M$ is obtained by viewing $\A/\cK^0$ as an infinite-dimensional hyperk\"ahler quotient where the Nahm equations play the r\^ole of the moment map. The metric comes from the inner-product
$$\ip{X,Y}=\sum_{i=0}^3\int_0^1\ip{X_i(t),Y_i(t)}dt$$
on the Banach space of $C^1$ maps $[0,1]\to\k^4$, where $\ip{\cdot,\cdot}$ is our choice of $K$-invariant inner-product on $\k$. The three complex structures $\I,\J,\K$ come from viewing this Banach space as the quaternionic vector space of $C^1$ maps $A_0+iA_1+jA_2+kA_3:[0,1]\to\k\otimes\H$.

There is also an action of $\SO(3)$ on $\M$ obtained by rotating $A_1,A_2,A_3$ while keeping $A_0$ fixed. This action preserves the metric but rotates the complex structures. In particular, $(\M,g,\I)\cong(\M,g,\J)\cong(\M,g,\K)$ as K\"ahler manifolds. These K\"ahler structures also have global potentials; for instance:

\begin{proposition}[{Dancer--Swann \cite[\S3]{dan96}}]
The map
$$
F_{12}:\M\too\R,\quad F_{12}(A)= \frac{1}{2}\int_0^1\|A_1(t)\|^2+\|A_2(t)\|^2\,dt
$$
is a K\"ahler potential for $(g,\I)$ and $(g,\J)$.\qed
\end{proposition}

More precisely, Dancer--Swann observed that $F_{12}$ is a moment map with respect to $\omega_\K$ for the action of the $\U(1)$ subgroup of $\SO(3)$ fixing $\K$ while rotating $\I$ and $\J$, and hence is a K\"ahler potential for $(g,\I)$ and $(g,\J)$ by \cite[\S3(E)]{hit87}. 

Similarly,
$$
F_{13}:\M\too\R,\quad F_{13}(A)= \frac{1}{2}\int_0^1\|A_1(t)\|^2+\|A_3(t)\|^2\,dt
$$
is a K\"ahler potential for $(g,\I)$ and $(g,\K)$, and hence both $F_{12}$ and $F_{13}$ are K\"ahler potentials for $(g,\I)$. These potentials are not proper, however, so they do not satisfy the assumptions of Theorem \ref{77svbtma} (which imply properness). But their average
\begin{equation}\label{9v5jg218}
F\coloneqq \frac{F_{12}+F_{13}}{2}:\M\too\R,\quad F(A)=\frac{1}{4}\int_0^12\|A_1(t)\|^2+\|A_2(t)\|^2+\|A_3(t)\|^2\,dt,
\end{equation}
which is still a K\"ahler potential for $(g,\I)$, \textit{does} satisfy these assumptions (when viewed as a map on $T^*G$), and this is one of the main technical results of this paper.

More precisely, we push $F$ to $T^*G$ via the isomorphism $\varphi$ of Theorem \ref{knvp0kum}, i.e.\ we let
$$f\coloneqq F\circ \varphi^{-1}:T^*G\too \R.$$
Then, our goal is to show that $\C[T^*G]\s o(e^f)$. Let us first define this notation more precisely:

\begin{definition}\label{18hr8xku}
Let $X$ be a topological space and $u,v:X\to\R$ two functions. We say that $v$ \textit{dominates} $u$, denoted $u\in o(v)$, if for all $c>0$ there exists a compact set $C\s X$ such that $|u(x)| \le c v(x)$ for all $x\in X\setminus C$.
\end{definition}

This generalizes the familiar notion of ``little-o'' for functions $\R\to\R$. We will show:

\begin{proposition}\label{h6lxerau}
$\C[T^*G]\s o(e^f)$
\end{proposition}

Moreover, it follows directly from the definition of the $K\times K$-action on $\M$ that $F$ is $K\times K$-invariant, and hence $f$ is a K\"ahler potential on $T^*G$ satisfying the assumptions of Theorem \ref{77svbtma} with respect to the $K\times K$-action (since $\varphi$ is a $K\times K$-equivariant biholomorphism).

Thus, we can identify certain symplectic quotients of $T^*G$ with GIT quotients using Theorem \ref{77svbtma}. But, as explained above, our main concern is rather to describe the \emph{hyperk\"ahler} quotients of $T^*G$. To explain this, we first recall that there is a canonical hyperk\"ahler moment map for the action of $K\times K$ on $\M$:

\begin{proposition}[Dancer--Swann \cite{dan96}]\label{rupl8lue}
By identifying $\k^*$ with $\k$ by the $K$-invariant inner-product, the map
$$\mu:\M\too(\k^*\times\k^*)^3,\quad A\mtoo\begin{pmatrix} \hspace{7.5pt} A_1(0) & \hspace{7.5pt}A_2(0) & \hspace{7.5pt}A_3(0) \\
-A_1(1) & -A_2(1) & -A_3(1)
\end{pmatrix}
$$
is a hyperk\"ahler moment map for the action of $K\times K$ on $\M$. Moreover, under the isomorphism $\varphi:\M\to T^*G$, the complex part $\mu_\C\coloneqq\mu_\J+i\mu_\K$ coincides with $\Phi$ in \eqref{2jztwdve}, i.e.\ $\mu_\C=\Phi\circ\varphi$.\qed\end{proposition}

In order to apply Theorem \ref{77svbtma}, we will show that the real part $\mu_\I$ of this hyperk\"ahler moment map is the moment map associated to the K\"ahler potential $f$ on $T^*G$:

\begin{proposition}\label{mtlkhb8t}
We have $\mu_\I(A)(Z)=dF(\I Z^\#_A)$ for all $A\in \M$ and $Z\in\k\times\k$. 
\end{proposition}

For any closed subgroup $L\s K\times K$, we denote by
$$\mu_{\l}\coloneqq(\mu_\I|_\l,\mu_\J|_\l,\mu_\K|_\l):\M\too\l^*\times\l^*\times\l^*$$
the induced hyperk\"ahler moment map for the action of $L$ on $\M$. Then, we have the following refinement of Theorem \ref{afdiktv2}:

\begin{theorem}\label{p8uaulmf}
Let $L$ be a closed subgroup of $K\times K$, let $\chi:L\to S^1$ be a character, let $\xi_\J,\xi_\K\in\l^*\coloneqq\Lie(L)^*$ be central, let $\xi\coloneqq(ai\,d\chi,\xi_\J,\xi_\K)\in\l^*\otimes\R^3$ for $a>0$, and let $\eta\coloneqq\xi_\J+i\xi_\K\in\l_\C^*$. Let $H\coloneqq L_\C$ and let $\Phi_\h:T^*G\to\h^*$ be the restriction of \eqref{2jztwdve} to $\h\coloneqq\Lie(H)\s\g\times\g$. Then, the diffeomorphism $\varphi:\M\to T^*G$ maps $\mu^{-1}_\l(\xi)$ to $\Phi^{-1}_\h(\eta)\ass{\chi}$ and the restriction $\mu_\l^{-1}(\xi)\to\Phi_\h^{-1}(\eta)\ass{\chi}$ descends to a homeomorphism
$$\mu_\l^{-1}(\xi)/L\too \Phi^{-1}_\h(\eta)\sll{\chi}H.$$
Moreover, $L$ acts freely on $\mu^{-1}(\xi)$ if and only if $H$ acts freely on $\Phi_\h^{-1}(\eta)\ass{\chi}$ and, in that case, the homeomorphism is an isomorphism of complex-symplectic manifolds.
\end{theorem}

Here, of course, the complex-symplectic structure on $\mu^{-1}_\l(\xi)/L$ is $(\I,\omega_\J+i\omega_\K)$ and the one on $\Phi_\h^{-1}(\xi_\C)\sll{\chi}H$ comes from the complex-symplectic Marsden--Weinstein theorem.

\begin{remark}
When the action of $L$ is not necessarily free, $\mu_\l^{-1}(\xi)/L$ is a stratified hyperk\"ahler space and the homeomorphism $\mu_\l^{-1}(\xi)/L\to\Phi^{-1}(\xi_\C)\sll{\chi}H$ is an isomorphism of stratified complex-symplectic spaces. This follows immediately from the results of this paper together with the author's paper \cite{may18}. See also \cite[Theorem 3.1]{may17} for the case where $\chi=1$.
\end{remark}

\subsection{Organization of the paper}

In \S\ref{lo2u1e7t} we prove the Kempf--Ness type theorem (Theorem \ref{77svbtma}). In \S\ref{phz5zi7d} we prove the results about the K\"ahler potential on $T^*G$ (Proposition \ref{h6lxerau} and Proposition \ref{mtlkhb8t}) and deduce the result about the hyperk\"ahler quotients of $T^*G$ (Theorem \ref{p8uaulmf}).  

\subsection{Acknowledgements}

I thank Andrew Dancer (my PhD supervisor), Frances Kirwan, and Kevin McGerty for useful discussions and comments.

\section{Kempf--Ness type theorems}\label{lo2u1e7t}

The goal of this section is to prove the generalization of the affine Kempf--Ness theorem mentioned in the introduction, i.e.\ Theorem \ref{77svbtma}. To do this, we first get a general form of the Kempf--Ness theorem for complex algebraic varieties with an integral K\"ahler form (Theorem \ref{aaqv7vrm}). This is mainly an adaptation of the original work of Kempf--Ness \cite{kem79} in the context of polarized K\"ahler manifolds as in Guillemin--Sternberg \cite[\S4]{gui82-geo} and using results of Heinzner--Loose \cite{hei94} (Theorem \ref{9ez7t5ob} below) to omit the compactness assumption. We then apply it to affine varieties using ideas of King \cite{kin94} to get Theorem \ref{77svbtma}.

\subsection{Symplectic reductions and analytic Hilbert quotients}\label{6v1119jt}

We first review some general results about the relationship between symplectic reduction and quotients of complex analytic spaces. Other expositions can be found in \cite{gre10,hei01,hei00,may18}; see also \cite{hei96,hei98,hhl94}. The goal is to recall the link between the following two notions:

\begin{definition}
A \textit{Hamiltonian K\"ahler manifold} is a triple $(M,K,\mu)$ where $M$ is a K\"ahler manifold, $K$ is a compact Lie group acting on $M$ by preserving the K\"ahler structure, and $\mu:M\to\k^*$ is a moment map for the action of $K$ on $M$ with respect to the K\"ahler form.
\end{definition}

\begin{definition}\label{q6po9pgr}
Let $(X,\O_X)$ be a complex analytic space and $G$ a complex reductive group acting holomorphically on $X$. An \textit{analytic Hilbert quotient} of $X$ by $G$ is a complex analytic space $(Y,\O_Y)$ together with a $G$-invariant surjective holomorphic map $\pi:X\to Y$ such that:
\begin{itemize}
\item[(i)] the map $\pi:X\to Y$ is \textit{locally Stein}, i.e.\ $Y$ has a cover by Stein open sets whose preimages are Stein;
\item[(ii)] $\O_Y=(\pi_*\O_X)^G$.
\end{itemize}
\end{definition}

Analytic Hilbert quotients are also sometimes called \textit{semistable quotients} \cite{hei98}. An important property is that they are categorical quotients for complex analytic spaces. In particular, they are unique up to biholomorphisms. We can think of them as the analytic analogue of good quotients in algebraic geometry.

The relationship between the above two definitions is stated in the following theorem, which can be thought of as a very general analytic version of the Kempf--Ness theorem:

\begin{theorem}[{Heinzner--Loose \cite{hei94} (see also \cite[\S0]{hei96})}]\label{9ez7t5ob}
Let $(M,K,\mu)$ be a Hamiltonian K\"ahler manifold such that the action of $K$ extends to a holomorphic action of $G\coloneqq K_\C$. Let
$$M\ass{\mu}\coloneqq\{p\in M:\overline{G\cdot p}\cap\mu^{-1}(0)\ne\emptyset\}$$
and let $M\ass{\mu}\sll{}G\coloneqq M\ass{\mu}/{\sim}$ where
$$p\sim q\iff\overline{G\cdot p}\cap\overline{G\cdot q}\cap M\ass{\mu}\ne\emptyset.$$
Let $\O$ be the sheaf of continuous functions on $M\ass{\mu}\sll{}G$ whose pullback to $M\ass{\mu}$ \textup{(}which is open\textup{)} is holomorphic. Then, $(M\ass{\mu}\sll{}G,\O)$ is a complex analytic space and the quotient map $M\ass{\mu}\to M\ass{\mu}\sll{}G$ is an analytic Hilbert quotient. Let
\begin{equation}\label{uvqxi87b}
M\aps{\mu}\coloneqq\{p\in M\ass{\mu}:G\cdot p\text{ is closed in }M\ass{\mu}\}.
\end{equation}
Then,
\begin{equation}\label{siczshzr}
p\in M\aps{\mu}\iff G\cdot p\cap\mu^{-1}(0)\ne\emptyset.
\end{equation}
Moreover, the inclusion $\mu^{-1}(0)\s M\ass{\mu}$ descends to a homeomorphism
\begin{equation*}
\mu^{-1}(0)/K \overset{\cong}{\too} M\ass{\mu}\sll{}G.\tag*{\qed}
\end{equation*}
\end{theorem}

In general, $\mu^{-1}(0)/K$ is a stratified symplectic space \cite{sja91} and Sjamaar \cite[Theorem 2.10]{sja95} showed the symplectic form on each stratum is K\"ahler. For the purpose of this paper, we want to emphasize the special case of free actions:

\begin{proposition}
Let $(M,K,\mu)$ and $G$ be as in \textup{Theorem \ref{9ez7t5ob}}. Then, $K$ acts freely on $\mu^{-1}(0)$ if and only if $G$ acts freely on $M\ass{\mu}$. In that case, $\mu^{-1}(0)/K$ and $M\ass{\mu}\sll{}G$ are smooth and the reduced symplectic form on $\mu^{-1}(0)/K$ is a K\"ahler form on $M\ass{\mu}\sll{}G$.\qed
\end{proposition}

\subsection{GIT quotients and Kempf--Ness type theorems}

Another large class of examples of analytic Hilbert quotients comes from algebraic geometry, as we now explain. Recall that when a complex reductive group $G$ acts on a complex algebraic variety $M$, a \textit{linearization} is a line bundle $\L$ on $M$ together with an action of $G$ by bundle automorphisms covering the original action on $M$. The set $M\ass{\L}$ of \textit{semistable} points is the union of the sets of the form $M_\sigma\coloneqq\{p\in M:\sigma(p)\ne 0\}$ for all $\sigma\in \bigoplus_{n=1}^\infty H^0(M,\L^n)$ such that $M_\sigma$ is affine. The set $M\ass{\L}$ is $G$-invariant and Zariski-open in $M$ and the main theorem of Geometric Invariant Theory (GIT) \cite{mum94} is that there exists a good quotient for the action of $G$ on $M\ass{\L}$, denoted $M\sll{\L}G$. The morphism $M\ass{\L}\to M\sll{\L}G$ is, in particular, a categorical quotient in the category of complex algebraic varieties. It is also an analytic Hilbert quotient:

\begin{proposition}
The analytification of a GIT quotient is an analytic Hilbert quotient.
\end{proposition}

\begin{proof}
The affine case is proved in \cite[\S6.4]{hei91} and the general case follows from the fact that a GIT quotient is constructed by gluing affine quotients.
\end{proof}

Now, an honest ``Kempf--Ness type theorem'' is a statement which identifies a symplectic quotient $\mu^{-1}(0)/K$ with a GIT quotient $M\sll{\L}G$ for some linearization $\L$. By the above results, it suffices to find a linearization $\L$ and a moment map $\mu$ such that $M\ass{\L}=M\ass{\mu}$. Indeed, since $M\sll{\L}G$ is an analytic Hilbert quotient and categorical quotients are unique, we get that $M\ass{\mu}\sll{}G=M\sll{\L}G$ and hence $\mu^{-1}(0)/K\cong M\sll{\L}G$. We summarize this in the following proposition:

\begin{proposition}\label{p1vjaeuf}
Let $(M,K,\mu)$ be a Hamiltonian K\"ahler manifold where $M$ is also a complex algebraic variety. Suppose that the $K$-action extends to a complex algebraic action of $G\coloneqq K_\C$. If $\L$ is a linearization of the $G$-action such that $M\ass{\L}=M\ass{\mu}$ then $\mu^{-1}(0)\s M\ass{\L}$ and this inclusion descends to a homeomorphism $\mu^{-1}(0)/K\to M\sll{\L}G$. Moreover, $K$ acts freely on $\mu^{-1}(0)$ if and only if $G$ acts freely on $M\ass{\L}$ and, in that case, the reduced symplectic form on $\mu^{-1}(0)/K$ is a K\"ahler form on $M\sll{\L}G$.\qed
\end{proposition}

Thus, a Kempf--Ness type theorem is a statement of the form $M\ass{\L}=M\ass{\mu}$, i.e.\ $\textit{algebraic semistability} = \textit{analytic semistability}.$
This brings naturally the question of how to get moment maps from linearizations. There is a well-known process for this, which we review next:

\subsection{Moment maps from linearizations}\label{exlp1xk4}

Let $M$ be a complex manifold and $\L$ a unitary line bundle on $M$, i.e.\ a holomorphic line bundle endowed with a hermitian fibre metric. Recall that $\L$ is called \textit{positive} if its curvature $\Theta_{\L}\in\Omega^2(M;i\R)$ has the property that the real $(1,1)$-form $\frac{i}{2\pi}\Theta_\L$ is K\"ahler. Since $\frac{i}{2\pi}\Theta_\L$ is a representative of the first Chern class of $\L$, a K\"ahler form $\omega$ on $M$ is equal to $\frac{i}{2\pi}\Theta_\L$ for some $\L$ if and only if $\omega$ is \textit{integral}, i.e.\ its cohomology class is in the image of the natural map $H^2(M;\Z)\to H^2(M;\R)$. A \textit{polarized K\"ahler manifold} is a triple $(M,\omega,\L)$ where $\omega$ is a K\"ahler form and $\L$ is a positive line bundle such that $\omega=\frac{i}{2\pi\hbar}\Theta_\L$ for some $\hbar>0$. The positive real number $\hbar$ will be fixed throughout this section.

Let $(M,\omega,\L)$ be a polarized K\"ahler manifold and $G$ a complex reductive group acting on $\L$ by bundle automorphisms in such a way that a maximal compact subgroup $K\s G$ preserves the hermitian fibre metric. This implies that the K\"ahler form $\omega= \frac{i}{2\pi\hbar}\Theta_\L$ is $K$-invariant. Moreover, there is a canonical moment map for the action of $K$ on $M$ with respect to $\omega$. This is well-known in the literature, but we will provide a proof for completeness and for fixing the notation; see e.g.\ Guillemin--Sternberg \cite[\S3]{gui82-geo}, Woodward \cite[Proposition 3.2.9]{woo11}, Thomas \cite[\S4]{tho06}, Donaldson--Kronheimer \cite[\S6.5.1]{don90}, or Sjamaar \cite[\S2.2]{sja95}.

Let $\L^*$ be the line bundle dual to $\L$. Then, $\L^*$ inherits a hermitian metric from $\L$ and also an action of $G$. Let $\dot{\L}^*$ be the complement of the zero-section in $\L^*$ and define $\hat{\mu}:\dot{\L}^*\to\k^*$ by
\begin{equation}\label{govj0puo}
\hat{\mu}(v)(X)=\frac{d}{dt}\Big|_{t=0}\frac{1}{4\pi\hbar}\log\|\exp(ti X)\cdot v\|^2,\quad(v\in \dot{\L}^*,X\in\k).
\end{equation}
Then, $\hat{\mu}$ is $\C^*$-invariant and hence descends to a smooth map $\mu:M\to\k^*$. 

\begin{proposition}\label{a9nvt66k}
The map $\mu$ is a moment map for the action of $K$ on $M$ with respect to $\omega$.
\end{proposition}

\begin{proof}
Let $\rho:\dot{\L}^*\to M$ be the projection and let
$$\varphi:\dot{\L}^*\too\R,\quad \varphi(v)=\log\|v\|^2.$$
The curvature of $\L^*$ is $-\Theta_{\L}$, so (by definition of the canonical connection on a unitary line bundle) we have $\rho^*(-\Theta_{\L})=\db\d\varphi$ and hence
\begin{equation}\label{h86vsbzw}
\db\d\varphi=2\pi\hbar i\rho^*\omega.
\end{equation}
Let $X\in\k$; we want to show that $d\mu^X=X^\#\intp\omega$. By definition of $\mu$, we have $(\mu^X\circ\rho)(v)=\frac{1}{4\pi\hbar}d\varphi(\I X^\#_v)$. Also, since $\varphi$ is $K$-invariant, we have $d\varphi(X^\#)=0$, so
\begin{equation}\label{ku7cgx6d}
\d\varphi(X^\#)=-\frac{i}{2}d\varphi(\I X^\#)=-2\pi\hbar i\mu^X\circ\rho.
\end{equation}
Now, since $K$ acts by biholomorphisms on $\L^*$, the Lie derivative $\LD_{X^\#}$ commutes with $\d$ and hence $\LD_{X^\#}\d\varphi=\d(\LD_{X^\#}\varphi)=0$. Thus, using Cartan's magic formula, \eqref{h86vsbzw}, and \eqref{ku7cgx6d}, we get
\begin{align*}
0&=d(X^\#\intp \d\varphi)+X^\#\intp d\d\varphi \\
&=d(\d\varphi(X^\#)) + X^\#\intp\db\d\varphi \\
&= d(-2\pi\hbar i \mu^X\circ\rho) + X^\#\intp 2\pi \hbar i\rho^*\omega \\
&= 2 \pi\hbar i\rho^*\left(- d\mu^X + X^\#\intp\omega \right).
\end{align*}
Hence, $d\mu^X=X^\#\intp \omega$.

To show equivariance, let $p\in M$ and take a point $\hat{p}\in \dot{\L}^*$ above $p$. We want to show that $\mu(k\cdot p)(X)=\mu(p)(\Ad_{k^{-1}}X)$ for all $k\in K$ and $X\in\k$, or equivalently,
$$\frac{d}{dt}\Big|_{t=0}\log\|\exp(ti X)\cdot k\cdot\hat{p}\|^2=\frac{d}{dt}\Big|_{t=0}\log\|\exp(ti\Ad_{k^{-1}}X)\cdot\hat{p}\|^2.$$
This is clear since $\exp(ti\Ad_{k^{-1}}X)=k^{-1}\exp(ti X)k$ and the fibre metric is $K$-invariant. 
\end{proof}

Therefore, $(M,K,\mu)$ is an integrable Hamiltonian K\"ahler manifold. Let us fix a terminology for the kind of Hamiltonian manifolds obtained in this way:

\begin{definition}
A \textit{polarized Hamiltonian K\"ahler manifold} is a tuple $(M,\omega, K,\L,\mu)$ where $(M,\omega,\L)$ is a polarized K\"ahler manifold, $K$ is a compact Lie group acting on $\L$ by preserving the unitary structure and extending to an action of $G=K_\C$, and $\mu$ is the associated moment map as in Proposition \ref{a9nvt66k}.
\end{definition}

Given a character $\chi:G\to\C^*$, we can \textit{twist} the action of $G$ on $\L$ by defining a new action $(g, v)\mto \chi(g)\,g\cdot v$ on $\L$. We denote by $\L_\chi$ the line bundle $\L$ with this twisted $G$-action. The action of $K$ on $\L_\chi$ still preserves the fibre metric so the discussion above implies that we have a new moment map $\mu_\chi$ for the action of $K$ on $M$. The next proposition identifies this map.

\begin{proposition}\label{wt4uzwtl}
Let $(M,\omega,K,\L,\mu)$ be a polarized Hamiltonian K\"ahler manifold and $\chi:G\to \C^*$ a character. Then, the moment map associated to $\L_\chi$ is the shift $\mu_\xi\coloneqq\mu-\xi$ where $\xi \coloneqq\frac{i}{2\pi\hbar}d\chi\in\k^*$. Thus, $(M,\omega,K,\L_\chi,\mu_\xi)$ is a polarized Hamiltonian K\"ahler manifold.
\end{proposition}

\begin{remark}
Since $\chi(K)$ is a compact subgroup of $\C^*$ we have $\chi(K)\s S^1$. Hence, $d\chi(\k)\s i\R$ and $\frac{i}{2\pi\hbar}d\chi\in\k^*$. Moreover, $\frac{i}{2\pi\hbar}d\chi$ is central.
\end{remark}

\begin{proof}
Let $\hat{\mu}_\chi: \dot{\L}_\chi^*  \to\k^*$ be the lift of the moment map associated to $\L_\chi$ as in \eqref{govj0puo}. The new action of $G$ on $\L^*$ is $(g,v)\mto\chi(g)^{-1} g\cdot v$, so for all $X\in\k$ and $v\in \dot{\L}^*$, we have
\begin{align*}
\hat{\mu}_\chi(v)(X) &= \frac{d}{dt}\Big|_{t=0}\frac{1}{4\pi\hbar}\log\|\chi(\exp(i tX))^{-1}\exp(i tX)\cdot v\|^2 \\
&= \frac{d}{dt}\Big|_{t=0}\frac{1}{4\pi\hbar}\left(-\log|\chi(\exp(i tX))|^2+\log\|\exp(i tX)\cdot v\|^2 \right) \\
&= \hat{\mu}(v)(X) -\frac{d}{dt}\Big|_{t=0}\frac{1}{4\pi\hbar}|\chi(\exp(i tX))|^{2} \\
&=\hat{\mu}(v)(X) - \frac{1}{4\pi\hbar}(d\chi(i X)+\overline{d\chi(i X)}) \\
&=\hat{\mu}(v)(X) - \frac{i}{2\pi\hbar}d\chi(X).
\end{align*}
Hence $\mu_\chi=\mu-\xi$. 
\end{proof}

\subsection{The general Kempf--Ness theorem}\label{5cku0gmt}

In this subsection, we state and prove a general form of the Kempf--Ness theorem for polarized Hamiltonian K\"ahler manifolds. The following definition will be one of the assumptions of the theorem:

\begin{definition}
Let $M$ be a complex algebraic variety with an algebraic action of a complex reductive group $G$ and let $\L$ be a linearization of this action. We say that $(M,G, \L)$ satisfies the \textit{geometric criterion} if a point $p\in M$ is $\L$-semistable if and only if for any non-zero lift $\hat{p}$ of $p$ in $\L^*$, the closure of the orbit $G\cdot \hat{p}\s \L^*$ is disjoint from the zero-section. 
\end{definition}

This is motivated by the following two examples:

\begin{example}
Let $M\s\CP^n$ be a projective variety with a $G$-action coming from a representation $G\to\GL(n+1,\C)$ and let $\L=\O_{\CP^n}(1)|_M$ with the natural $G$-action. It is a standard result that $(M,G,\L)$ satisfies the geometric criterion \cite[Proposition 2.2]{mum94}.
\end{example}

\begin{example}\label{35g14u8q}
Let $M$ be a complex affine variety with an algebraic action of $G$ and let $\L_\chi$ be the trivial line bundle on $M$ with the $G$-action twisted by a character $\chi$. Then, as observed by King \cite[Lemma 2.2]{kin94}, $(M,G,\L_\chi)$ also satisfies the geometric criterion.
\end{example}

We will only use the affine case, but it is interesting to know that these two examples are special cases of a more general result for projective-over-affine varieties (see e.g.\ \cite[\S4]{gul15} or \cite[\S1]{pro06}).

Our goal is to prove the following Kempf--Ness type theorem.

\begin{theorem}\label{aaqv7vrm} 
Let $(M,\omega,K,\L,\mu)$ be a polarized Hamiltonian K\"ahler manifold such that $(M,G,\L)$ is complex algebraic and satisfies the geometric criterion. Suppose also that the norm-squared function $\L^*\to\R$, $v\mto\|v\|^2$ is proper on all closed $G$-orbits that are disjoint from the zero section. Then, $M\ass{\L}=M\ass{\mu}$.
\end{theorem}

This result might be known to some experts, but we have not found a proof nor a statement at this level of generality. Special cases are found in \cite{nes84,kir84,gui82-geo,sja95,kin94,aza93}; see also \cite{don90,kir98,tho06,woo11} for good expositions. For completeness, we provide a full proof of Theorem \ref{aaqv7vrm}. The proof resembles the original one of Kempf--Ness \cite{kem79}, but in a more abstract setting. We could have obtained the affine case (Theorem \ref{77svbtma}) more directly, but this more general result might be of independent interest and the proof does not require much more effort.

The rest of this subsection is devoted to the proof of Theorem \ref{aaqv7vrm}. For $v\in\dot{\L}^*$, define the \emph{Kempf--Ness function}
$$F_v: G \too \R,\quad F_v(g) = \frac{1}{4\pi\hbar}\log\|g\cdot v\|^2.$$
Let $p\in M$ and fix a non-zero lift $\hat{p}$ of $p$ in $\L^*$. Our first goal is:

\begin{proposition}\label{edyq7cpy}
The following are equivalent:
\begin{itemize}
\item[\textup{(1)}] $p\in M\aps{\mu}$ \textup{(}see \eqref{uvqxi87b}\textup{)};
\item[\textup{(2)}] $F_{\hat{p}}$ has a critical point;
\item[\textup{(3)}] $F_{\hat{p}}$ has a global minimum;
\item[\textup{(4)}] $G\cdot\hat{p}$ is closed in $\L^*$.
\end{itemize}
\end{proposition}

We will show
$$(4)\implies(1)\implies(2)\implies(3)\implies(4).$$

\begin{lemma}\label{a6x2akvq}
For all $X\in\k$, we have
\begin{align}
\frac{d}{dt}F_{\hat{p}}(e^{tiX}) &= \mu(e^{tiX}\cdot p)(X) \label{ps6x1fn6}\\
\frac{d^2}{dt^2}F_{\hat{p}}(e^{tiX}) &= \|X^\#_{e^{tiX}\cdot p}\|^2,\label{wo2rijm8}
\end{align}
where the latter is the norm of $X^\#_{e^{tiX}\cdot p}$ with respect to the K\"ahler metric on $M$. In particular,
\begin{equation}\label{wv88dvb1}
(d F_{\hat{p}})_1(iX)=\mu(p)(X).
\end{equation}
\end{lemma}

\begin{proof}
This is a simple computation. The first identity follows from the definition of $\mu$ and the fact that $F_{\hat{p}}(e^{(t_0+t)iX})=F_{e^{t_0iX}\cdot\hat{p}}(e^{tiX})$. To get the second identity, note that
\begin{align*}
\frac{d^2}{dt^2}\Big|_{t=t_0}F_{\hat{p}}(e^{tiX}) &= \frac{d}{dt}\Big|_{t=t_0}\mu(e^{tiX}\cdot p)(X) = \frac{d}{dt}\Big|_{t=0}\mu^X(e^{tiX}\cdot e^{t_0iX}\cdot p) \\
&= d\mu^X(\I X^\#_{e^{t_0iX}\cdot p}) =  \omega(X^\#_{e^{t_0iX}\cdot p},\I X^\#_{e^{t_0iX}\cdot p}) \\
&=  \|X^\#_{e^{t_0iX}\cdot p}\|^2.\qedhere
\end{align*}
\end{proof}

\begin{lemma}
$(4)\Rightarrow(1)$, i.e.\ if $G\cdot\hat{p}$ is closed in $\L^*$ then $p\in M\aps{\mu}$.
\end{lemma}

\begin{proof}
Since $G\cdot\hat{p}$ is closed, $\|\cdot\|^2$ is proper on $G\cdot\hat{p}$ by assumption. Hence, $\|G\cdot\hat{p}\|^2$ is closed in $\R$, so it attains a minimum $\|g\cdot\hat{p}\|^2$ for some $g\in G$. Then, $g$ is a minimum of $F_{\hat{p}}$ so $(dF_{\hat{p}})_{g}=0$. Note that $F_{g\cdot\hat{p}}=F_{\hat{p}}\circ R_g$ where $R_g:G\to G$ is right multiplication by $g$, so $(d F_{g\cdot\hat{p}})_1=(dF_{\hat{p}})_g\circ (dR_g)_1=0$. By \eqref{wv88dvb1}, this implies $\mu(g\cdot p)=0$ and hence $p\in M\aps{\mu}$ by \eqref{siczshzr}.
\end{proof}

\begin{lemma}
$(1)\Rightarrow(2)$, i.e.\ if $p\in M\aps{\mu}$ then $F_{\hat{p}}$ has a critical point.
\end{lemma}

\begin{proof}
By \eqref{siczshzr}, if $p\in M\aps{\mu}$ then there exists $g\in G$ such that $\mu(g\cdot p)=0$. Then, $(dF_{g\cdot\hat{p}})_1(i\k)=0$ by \eqref{wv88dvb1}. Also, $(dF_{g\cdot\hat{p}})_1(\k)=0$ since $F_{g\cdot\hat{p}}$ is $K$-invariant, so $(dF_{g\cdot\hat{p}})_1=0$. Since $F_{g\cdot\hat{p}}=F_{\hat{p}}\circ R_g$, we get $(dF_{\hat{p}})_g=0$. 
\end{proof}

Note that the group $G_p$ acts naturally on the fibre $\L^*_p$. Since $\L^*_p$ is one-dimension\-al, this action must be multiplication by a non-zero complex number, i.e.\ we have a character
$$\lambda_p:G_p\too \C^*$$
such that $g\cdot v=\lambda_p(g)v$ for all $v\in \L^*_p$ and $g\in G_p$. Let $\k_p\coloneqq\Lie(K_p)$ and $\g_p\coloneqq\Lie(G_p)$. Then, $i\k_p\s\g_p$ so from the polar decomposition we get that $(K_p)_\C\s G_p$.

\begin{lemma}\label{lmwfiwqe}
If $(dF_{\hat{p}})_1=0$ then $|\lambda_p(h)|=1$ for all $h\in (K_p)_\C$.
\end{lemma}

\begin{proof}
Let $h\in (K_p)_\C$. We have $\|h\cdot\hat{p}\|=\|\lambda_p(h)\hat{p}\|=|\lambda_p(h)|\|\hat{p}\|$ so it suffices to show that $\|h\cdot\hat{p}\|=\|\hat{p}\|$. Write $h=ke^{iX}$ for $k\in K_p$ and $X\in\k_p$. Then, $\|h\cdot\hat{p}\|=\|e^{iX}\cdot\hat{p}\|$. Thus, by \eqref{ps6x1fn6} and \eqref{wv88dvb1},
$$\frac{d}{dt}\frac{1}{4\pi\hbar}\log\|e^{tiX}\cdot\hat{p}\|^2=\mu(e^{tiX}\cdot p)(X) = \mu(p)(X)=(dF_{\hat{p}})_1(iX)=0,$$
so $\|e^{tiX}\cdot\hat{p}\|^2$ is independent of $t$. In particular, $\|h\cdot\hat{p}\|=\|e^{iX}\cdot\hat{p}\|=\|\hat{p}\|$.
\end{proof}

The polar decomposition $K\times\k\cong G$ has a generalization due to Mostow \cite{mos55,mos55b} which we will need here (see also \cite[Theorem VI.1.4]{hel78}, \cite[Corollary 9.5]{hei07} or \cite[Theorem 2.4.9]{hei01}).

\begin{proposition}[Mostow Decomposition]\label{p7hf1zlt}
Let $K$ be a compact Lie group and $H\s K$ a closed subgroup. Let $\k\coloneqq\Lie(K)$, $\h\coloneqq\Lie(H)$, and let $\h^\perp$ be the orthogonal complement to $\h$ in $\k$ with respect to a $K$-invariant inner-product. Let $K\times_H\h^\perp$ be the quotient of $K\times\h^\perp$ by the $H$-action $h\cdot(k,X)=(kh^{-1},\Ad_kX)$. Then, the map
$$K\times_H\h^\perp\too K_\C/H_\C,\quad [k,X]\mtoo ke^{iX}H_\C$$
is a diffeomorphism.\qed
\end{proposition}

\begin{lemma}
$(2)\Rightarrow(3)$, i.e.\ if $F_{\hat{p}}$ has a critical point, then it has a global minimum.
\end{lemma}

\begin{proof}
Suppose that $(dF_{\hat{p}})_g=0$. By using that $F_{g\cdot\hat{p}}=F_{\hat{p}}\circ R_g$ and replacing $\hat{p}$ by $g\cdot\hat{p}$, we may assume without loss of generality that $(d F_{\hat{p}})_1=0$. For $X\in\k$, we have $\|X^\#_p\|^2=0$ if and only if $X^\#_p= 0$ if and only if $X\in\k_p=:\Lie(K_p)$. Thus, \eqref{wo2rijm8} implies that for all $X\in\k\setminus \k_p$ the function
$$F_{\hat{p},X}:\R\too\R,\quad t\mtoo F_{\hat{p}}(e^{tiX})$$
is convex on $\R$ and strictly convex near $t=0$. Moreover, since $(dF_{\hat{p}})_1=0$, we get that $t=0$ is a global minimum of $F_{\hat{p},X}$. Now, by Mostow decomposition (Proposition \ref{p7hf1zlt}) any $g\in G$ can be decomposed as $g=ke^{iX}h$ for $k\in K$, $X\in\k_p^\perp$ and $h\in (K_p)_\C\s G_p$. By Lemma \ref{lmwfiwqe}, we have $\|ke^{iX}h\cdot\hat{p}\|=\|e^{iX}\cdot\hat{p}\|$, so $F_{\hat{p}}(g)=F_{\hat{p}}(e^{iX})\ge F_{\hat{p}}(1)$, and hence $1$ is a global minimum of $F_{\hat{p}}$. 
\end{proof}

\begin{lemma}\label{ywqlnzs0}
Let $V$ be a finite-dimensional real vector space and $f:V\to\R$ a $C^2$ function. Suppose that $f(0)$ is a minimum of $f$ and that for all non-zero $v\in V$, the function $f_v:\R\to\R$, $f_v(t)\coloneqq f(tv)$ satisfies $f_v''(t)\ge0$ for all $t\in\R$ and $f_v''(0)>0$. Then, for any norm $\|\cdot\|$ on $V$, there are constants $c_0,c_1>0$ such that $\|v\|\le c_0+c_1f(v)$ for all $v\in V$. In particular, $f$ is proper.
\end{lemma}

\begin{proof}
The function $F:V\times\R\to\R$, $F(v,t)= f''_v(t)$ is continuous, $F(v,0)>0$ for all non-zero $v\in V$, and $S\coloneqq\{v\in V:\|v\|=1\}$ is compact, so there exist $\e>0$ and $\delta>0$ such that
$$f_v''(t)\geq \e,\quad\text{for all } t\in[0,\delta]\text{ and }v\in S.$$
Since $f(0)$ is a minimum, $f_v'(0)=0$ and hence for all $t\ge\delta$ and $v\in S$ we have
$$f_v'(t)=\int_0^t f_v''(s)ds=\int_0^\delta f_v''(s)ds+\int_\delta^tf_v''(s)ds\ge \int_0^\delta\e\, ds= \e\delta$$
so
$$f(tv)=f(\delta v)+\int_\delta^tf_v'(s)ds\geq f(\delta v)+\int_{\delta}^t\e\delta\,ds=f(\delta v)+(t-\delta)\e\delta.$$
Thus, for all $w\in V$ with $\|w\|\geq \delta$ we have $f(w)\geq f(\delta w/\|w\|)+(\|w\|-\delta)\e\delta$, or
$$\e\delta\|w\|\leq\e\delta^2+f(w)-f(\delta w/\|w\|).$$
Let $m=\sup\{|f(v)| : \|v\|\leq \delta\}$. Then, $-f(\delta w/\|w\|)\leq m$, so for all $w\in V$ such that $\|w\|\ge\delta$, we have
$$\e\delta\|w\|\leq \e\delta^2+m+f(w).$$
If $\|w\|\le\delta$ then this is also true since $|f(w)|\le m$ so $m+f(w)\ge 0$ and hence
$$\e\delta\|w\|\le\e\delta^2\le\e\delta^2+m+f(w).$$
Thus, 
$$\|v\|\leq \delta+\frac{m}{\e\delta}+\frac{1}{\e\delta}f(v),\quad\text{for all }v\in V,$$
and hence the first part of the lemma holds with $c_0=\delta+\frac{m}{\e\delta}$ and $c_1=\frac{1}{\e\delta}$. To show that $f$ is proper, let $C\s\R$ be compact. It suffices to show that $f^{-1}(C)$ is bounded. But $C\s[a,b]$ for some $a,b\in\R$, so if $v\in f^{-1}(C)$ then $f(v)\le b$ and hence $\|v\|\le c_0+c_1b$. 
\end{proof}

\begin{lemma}
$(3)\Rightarrow(4)$, i.e.\ if $F_{\hat{p}}$ has a global minimum, then $G\cdot\hat{p}$ is closed in $\L^*$. 
\end{lemma}

\begin{proof}
Let $\varphi:\dot{\L}^*\to\R$ be defined by $\varphi(v)=\frac{1}{4\pi\hbar}\log\|v\|^2=F_v(1)$. It suffices to show that the restriction of $\varphi$ to $G\cdot\hat{p}$ is proper (in general, if $\varphi:X\to\R$ is a continuous function on a metrizable topological space $X$ which is proper on a subset $A\s X$, then $A$ is closed). Without loss of generality, $F_{\hat{p}}$ has a global minimum at $1$. Define
$$f:\k_p^\perp\too\R,\quad f(X)=F_{\hat{p}}(e^{iX}).$$
Then, $f$ attains a global minimum at $0$. By Lemma \ref{ywqlnzs0} and Lemma \ref{a6x2akvq}, $f$ is proper. By Mostow's decomposition and Lemma \ref{lmwfiwqe}, the map
$$\alpha:S^1\times K\times\k_p^\perp\too G\cdot \hat{p},\quad (z,k,X)\mtoo zke^{iX}\cdot \hat{p}$$
is surjective. Since $S^1$ and $K$ are compact and $f$ is proper, the function $\psi:S^1\times K\times\k_p^\perp\to\R$, $\psi(z,k,X)=f(X)$ is also proper. Note that $\varphi\circ\alpha=\psi$. Since $\psi$ is proper and $\alpha$ is surjective, this implies that $\varphi$ is proper on $G\cdot\hat{p}$. Indeed, for all $C\s\R$, the surjectivity of $\alpha$ implies that $\varphi^{-1}(C)\cap G\cdot\hat{p}=\alpha(\psi^{-1}(C))$.
\end{proof}

This concludes the proof of Proposition \ref{edyq7cpy}. In particular, if $p\in\mu^{-1}(0)$ then $G\cdot\hat{p}$ is closed in $\L^*$ and hence, by the geometric criterion, $p\in M\ass{\L}$. Thus, the following general fact implies that $M\ass{\mu}\s M\ass{\L}$.

\begin{lemma}\label{fz959063}
Let $(M,K,\mu)$ be an integrable Hamiltonian K\"ahler manifold. Then, $M\ass{\mu}$ is the smallest $G$-invariant open set of $M$ containing $\mu^{-1}(0)$.
\end{lemma}

\begin{proof}
Let $U\s M$ be $G$-invariant, open, and containing $\mu^{-1}(0)$. We want to show that $M\ass{\mu}\s U$. Let $p\in M\ass{\mu}$. Then, by definition, there exists $q\in\overline{G\cdot p}\cap\mu^{-1}(0)$, where $q=\lim_{n\to\infty}g_n\cdot p$ for some $g_n\in G$. But $q\in\mu^{-1}(0)\s U$ and $U$ is open, so there exists $N\ge 0$ such that $g_n\cdot p\in U$ for all $n\ge N$. Then, $p\in U$ since $U$ is $G$-invariant. 
\end{proof}

We can now conclude the proof of Theorem \ref{aaqv7vrm}, i.e.\ that $M\ass{\L}=M\ass{\mu}$.

\begin{proof}[Proof of Theorem \ref{aaqv7vrm}]
As we just explained, Lemma \ref{fz959063} and Proposition \ref{edyq7cpy} imply that $M\ass{\mu}\s M\ass{\L}$. Now, let $p\in M\ass{\L}$. Then, $\overline{G\cdot\hat{p}}\s \L^*$ is disjoint from the zero section. Since the $G$-action on $\L$ is algebraic, there exists a closed orbit $G\cdot \hat{q}\s\overline{G\cdot\hat{p}}$ (see e.g.\ \cite[Proposition 1.8]{bor91}). Since $\overline{G\cdot\hat{p}}$ is disjoint from the zero section, $\hat{q}$ is non-zero and hence the point $q\in M$ below $\hat{q}$ is in $M\aps{\mu}$ by Proposition \ref{edyq7cpy}. Thus, $G\cdot q\cap\mu^{-1}(0)\ne\emptyset$. We have $q\in\overline{G\cdot p}$ so $\overline{G\cdot p}\cap\mu^{-1}(0)\ne\emptyset$ and hence $p\in M\ass{\mu}$. 
\end{proof}

\subsection{The affine Kempf--Ness theorem}\label{np59oshd}

We now come to the main goal of this section, which is to prove Theorem \ref{77svbtma}.

Let us first discuss Part (1). Let $M$ be a smooth complex affine variety and $K$ a compact Lie group acting on $M$. It is well known that if the action map $K\times M\to M$ is real algebraic, then this action extends to a complex algebraic action of $G\coloneqq K_\C$. This follows from the fact that for every $u\in\C[M]$, the linear span of $K\cdot u$ is finite-dimensional and hence $M$ can be embedded in a finite-dimensional complex representation $R$ of $K$. Then, we use the universality property of complexifications, which says that the representation $K\to\GL(R)$ extends uniquely to a representation $K_\C\to\GL(R)$ (the author learned this argument in \cite[p.\ 226]{hei00}).

Now, let $\chi:G\to\C^*$ be a character (equivalently, $\chi:K\to S^1$), let $\L$ be the trivial line bundle on $M$ with the trivial $G$-action $g\cdot(p,z)=(g\cdot p,z)$, and let $\L_\chi$ be $\L$ with the $G$-action twisted by $\chi$, i.e.\ $g\cdot(p,z)=(g\cdot p,\chi(g)z)$. Then, King \cite[\S2]{kin94} observed that the GIT quotient $M\sll{\L_\chi}G$ can be described as the variety $M\sll{\chi}G\coloneqq \Proj\bigoplus_{n=0}^\infty\C[M]^{\introred,\chi^n}$ and that $M\ass{\L_\chi}=M\ass{\chi}\coloneqq\{p\in M:\exists n\ge1,u\in\C[M]^{G,\chi^n}\text{ such that }u(p)\ne 0\}$. Here we recall that $\C[M]^{G,\chi^n}$ is the group of $u\in\C[M]$ such that $u(g\cdot p)=\chi(g)^nu(p)$ for all $g\in G$ and $p\in M$.

Let $f:M\to \R$ be a $K$-invariant K\"ahler potential on $M$ and let $\omega\coloneqq 2i\d\db f$ be the associated K\"ahler form. We want to apply Proposition \ref{a9nvt66k} to get a moment map, so we need a fibre metric on $\L$ such that $\omega=\frac{i}{2\pi \hbar}\Theta_{\L}$. Since $\L$ is trivial, we may define a metric by
$$\ip{(p,u),(p,v)}=u\bar{v} e^{-4\pi\hbar f(p)},\quad(p\in M,u,v\in\C)$$
where $\hbar>0$ is arbitrary.

\begin{lemma}
$\omega=\frac{i}{2\pi\hbar}\Theta_{\L}$
\end{lemma}

\begin{proof}
This follows directly from the definition of the canonical connection on a unitary line bundle (e.g.\ Wells \cite[\S III.2]{wel08}). Indeed, from \cite[p.\ 78, eq.\ (2.1)]{wel08}, the connection 1-form associated to the trivial frame is $\theta=e^{4\pi\hbar f}\d e^{-4\pi\hbar f}=-4\pi\hbar\d f$ and hence by \cite[Proposition 2.2]{wel08} the curvature is $\Theta_{\L}=\db\theta=-4\pi\hbar\db\d f$. Hence, $\frac{i}{2\pi\hbar}\Theta_{\L}=2i\d\db f=\omega$. 
\end{proof}

Thus, $(M,\omega, \L)$ is a polarized K\"ahler manifold. The fibre metric on $\L$ is $K$-invariant, so we have a moment map $\mu:M\to\k^*$ by Proposition \ref{a9nvt66k}. Explicitly:

\begin{lemma}\label{dpa42xux}
We have $\mu(p)(X)=df(\I X^\#_p)$ for all $p\in M$ and $X\in\k$.
\end{lemma}

\begin{proof}
The fibre metric on $\L^*=M\times\C$ is $\ip{(p,u),(p,v)}=u\bar{v}e^{4\pi\hbar f(p)}$. Every $p\in M$ has a canonical lift $\hat{p}=(p,1)$. Then, by definition of $\mu$, we have
\begin{align*}
\mu(p)(X) &= \ddt{0}\frac{1}{4\pi\hbar}\log\|(e^{tiX}\cdot p,1)\|^2 = \ddt{0}\frac{1}{4\pi\hbar}\log e^{4\pi\hbar f(e^{tiX}\cdot p)} \\
&= \ddt{0}f(e^{tiX}\cdot p) = df(\I X^\#_p).\qedhere
\end{align*}
\end{proof}

In particular, we proved Part (2) of Theorem \ref{77svbtma}.

Now, by Proposition \ref{wt4uzwtl}, the moment map associated to $\L_\chi$ is $\mu_\xi\coloneqq\mu-\xi$ where $\xi\coloneqq\frac{i}{2\pi\hbar}d\chi$. Hence, $(M,\omega,K,\L_\chi,\mu_\xi)$ is a polarized Hamiltonian K\"ahler manifold. Moreover, by Proposition \ref{p1vjaeuf}, to prove Theorem \ref{77svbtma} it suffices to show that if $\C[M]\s o(e^f)$ (Definition \ref{18hr8xku}) then $M\ass{\chi}=M\ass{\mu_\xi}$ (using $\hbar=\frac{1}{2\pi a}$). We do this using the general Kempf--Ness theorem (Theorem \ref{aaqv7vrm}). As explained earlier, $(M,G,\L)$ satisfies the geometric criterion \cite[Lemma 2.2]{kin94}. Hence, it suffices to show that, if $\C[M]\s o(e^f)$, then the norm-squared function
$$
N:\L^*_\chi=M\times\C\too\R,\quad (p,z)\mtoo |z|^2e^{4\pi\hbar f(p)}
$$
is proper on every closed $G$-orbit disjoint from the zero section. We prove this by the following two lemmas.

\begin{lemma}\label{3ruoz4bu}
If $\C[M]\s o(e^f)$ then $\C[M]\s o(e^{\alpha f})$ for all $\alpha>0$.
\end{lemma}

\begin{proof}
Let $c>0$ and let $u\in\C[M]$. We want to show that there exists a compact set $C\s M$ such that $|u(p)|\le ce^{\alpha f(p)}$ for all $p\in M\setminus C$. If $\alpha \geq 1$, this follows immediately from $\C[M]\s o(e^f)$. Hence, suppose that $0<\alpha<1$ and let $n\in\Z$ be such that $1<1/\alpha<n$. Since $1\in\C[M]\s o(e^f)$, there is a compact set $C_1\s M$ such that $1\leq c^{1/\alpha}e^{f(p)}$ for all $p\in M\setminus C_1$. Hence, $1\le ce^{\alpha f(p)}$ for all $p\in M\setminus C_1$. Now, $u^n\in\C[M]\s o(e^f)$ so there is a compact set $C_2\s M$ such that $|u(p)|^n\le c^{1/\alpha}e^{f(p)}$ for all $p\in M\setminus C_2$. Let $C\coloneqq C_1\cup C_2$ and let $p\in M\setminus C$. Then, either $|u(p)|\leq 1$ and hence $|u(p)|\le 1\le ce^{\alpha f(p)}$, or $|u(p)|>1$ and hence $|u(p)|=(|u(p)|^{1/\alpha})^\alpha<(|u(p)|^n)^{\alpha}\le (c^{1/\alpha}e^{f(p)})^{\alpha}=ce^{\alpha f(p)}$.
\end{proof}

\begin{lemma}
If $S\s M\times\C$ is Zariski-closed and disjoint from the zero-section, then $N|_S$ is proper.
\end{lemma}

\begin{proof}
The proof is adapted from King \cite[Lemma 6.3]{kin94}. Since $f$ is bounded below and the K\"ahler form $\omega=2i\d\bar{\d}f$ is unaffected by adding to $f$ a constant, we may assume that $f$ maps $M$ to $[0,\infty)$. To show that $N|_S$ is proper, it suffices to show that for all $B>0$ there exists a compact subset $C\s M$ and $r>0$ such that $N^{-1}([0,B])\cap S\s C\times \overline{\mathbb{D}}_r$, where $\overline{\mathbb{D}}_r$ is the closed disc of radius $r$ centred at $0$ in $\C$. Since $S$ is Zariski-closed and disjoint from $M\times\{0\}$ there is a polynomial $u\in\C[M\times\C]$ such that $u(M\times\{0\})=0$ and $u(S)=1$. Then, $u$ must be of the form
$$u(p,z)=zu_1(p)+\cdots+z^nu_n(p)$$
for some $u_k\in\C[M]$. By Lemma \ref{3ruoz4bu}, we have $u_k\in o(e^{2\pi\hbar kf})$ for all $k$, so there is a compact subset $C\s M$ such that
$$\frac{1}{2n}e^{2\pi\hbar kf(p)}>B^{k/2}|u_k(p)|$$
for all $p\in M\setminus C$ and $k=1,\ldots,n$. Now, let $(p,z)\in N^{-1}([0,B])\cap S$. Then, $|z|^2\leq |z|^2e^{4\pi\hbar f(p)}=N(p,z)\leq B$ so $z\in\overline{\mathbb{D}}_{\sqrt{B}}$. We claim that $p\in C$. Indeed, we have $N(p,z)=|z|^2e^{4\pi\hbar f(p)}\le B$ so $|z|^2\le e^{-4\pi\hbar f(p)}B$ and hence $|z|^k=(|z|^2)^{k/2}\le(e^{-4\pi\hbar f(p)}B)^{k/2}= e^{-2\pi\hbar k f(p)}B^{k/2}$. Thus, if $p\notin C$ we get
\begin{align*}
|u(p,z)| &\leq \sum_{k=1}^n |z|^k |u_k(p)|\leq \sum_{k=1}^n e^{-2\pi\hbar k f(p)}B^{k/2}|u_k(p)|< \sum_{k=1}^n\frac{1}{2n} = \frac{1}{2},
\end{align*}
contradicting that $(p,z)\in S$ since $u(S)=1$.
\end{proof}

By Theorem \ref{aaqv7vrm}, this implies that $M\ass{\L}=M\ass{\mu_\xi}$. Using Proposition \ref{p1vjaeuf}, this concludes the proof of Theorem \ref{77svbtma}.

\subsection{A counterexample}\label{xfc6kidf}

We show that the assumption on the relationship between $\C[M]$ and $f$ in Theorem \ref{77svbtma} is not superfluous. Namely, we give an example of an affine variety $M$ with an action of a complex reductive group $G=K_\C$ and a $K$-invariant K\"ahler potential $f$ that is proper and bounded below, but there exists a character $\chi$ such that if $\xi=i\,d\chi$ then $\mu^{-1}(\xi)/K$ and $M\sll{\chi}G$ are not homeomorphic. More precisely, we will have $\mu^{-1}(\xi)/K=\emptyset$ and $M\sll{\chi}G=\{\mathrm{pt}\}$. 

Consider $K=S^1$ acting on $M=\C^*$ by multiplication. An $S^1$-invariant K\"ahler potential on $\C^*$ is a function of the form
$$f:\C^*\too\R,\quad f(z)=r(|z|^2),$$
for any smooth function $r:(0,\infty)\to \R$ such that $t\ddot{r}(t)+\dot{r}(t)>0$ for all $t>0$. Indeed, we have
$$\omega\coloneqq 2i\d\db f= 4(t\ddot{r}(t)+\dot{r}(t))dx\wedge dy,$$
where $z=x+iy\in\C^*$ and $t=|z|^2$. Every character of $K=S^1$ is of the form
$$\chi:S^1\too S^1,\quad\chi(z)=z^n$$
for some $n\in\Z$. We consider the GIT quotient $M\sll{\chi}G=\C^*\sll{\L_n}\C^*$, where $\L_n$ is the trivial line bundle $\L_n=\C^*\times\C$ over $\C^*$ with the linearization $\lambda \cdot (z,u)=(\lambda z,\lambda^n u)$. We have $(\C^*)\ass{\L_n}=\C^*$, so $M\sll{\chi}G$ is a single point (for any $\chi$). 

The moment map associated to the potential $f$ as in Theorem \ref{77svbtma}(2) is
$$\mu:\C^*\too\R,\quad \mu(z)=-2|z|^2\dot{r}(|z|^2),$$
using the isomorphism $\k^*=(i\R)^*\to\R:\xi\mto\xi(i)$. Note that under this isomorphism $\k^*\cong\R$, the central element $\xi\coloneqq i\,d\chi\in\k^*$ is $-n\in\R$. Thus,
$$\mu^{-1}(\xi)/K=\{t\in(0,\infty):2t\dot{r}(t)=n\}.$$
Therefore, to get a counterexample, it suffices to find a proper function $r:(0,1)\to\R$ such that $t\ddot{r}(t)+\dot{r}(t)>0$ and $2t\dot{r}(t)$ is bounded. An example of such a function is
$$r(t)=\sqrt{1+(\log t)^2}.$$
Indeed, we have
$$t\ddot{r}(t)+\dot{r}(t) = \frac{1}{t(1+(\log t)^2)^{3/2}}\quad\text{and}\quad2t\dot{r}(t)=\frac{2\log t}{\sqrt{1+(\log t)^2}}.$$
Since $|2t\dot{r}(t)|\le 2$, we obtain a counterexample with $n=3$. Note that not every polynomial is in $o(e^f)$. For example, $e^{\sqrt{1+(\log|z|^2)^2}}<|z|^3$ for large $z$ so the polynomial $z^3\in\C[\C^*]$ is not in $o(e^f)$.

\subsection{Application to hyperk\"ahler quotients}\label{gcz55544}

Let $M$ be a smooth manifold with a hyperk\"ahler structure $(g,\I,\J,\K)$ invariant under the action of a compact Lie group $K$ and with a hyperk\"ahler moment map $\mu:M\to\k^*\otimes\R^3$. Let $\omega_\I,\omega_\J,\omega_\K$ be the K\"ahler forms associated to $\I,\J,\K$. Suppose that the action of $K$ extends to an $\I$-holomorphic action of $G\coloneqq K_\C$. Recall \cite{hit87} that $\omega_\J+i\omega_\K$ is a $G$-invariant complex-symplectic form on $(M,\I)$ and $\mu_\C\coloneqq\mu_\J+i\mu_\K:M\to\g^*\eqqcolon\Lie(G)^*$ is a moment map for the action of $G$ on $(M,\I,\omega_\J+i\omega_\K)$. Suppose that $(M,\I)$ is the analytification of a complex affine variety and that the $G$-action and $\mu_\C$ are complex algebraic.  Suppose also that there is a K\"ahler potential $f:M\to\R$ for $(g,\I,\omega_\I)$ such that $\C[M]\s o(e^f)$ and $\mu_\I(p)(X)=df(\I X^\#_p)$ for all $p\in M$ and $X\in\k$. Let $\chi:K\to S^1$ be a character, let $\xi_\I\coloneqq ai\, d\chi$ for $a>0$, let $\xi_\J,\xi_\K\in\k^*$ be central, let $\xi\coloneqq(\xi_\I,\xi_\J,\xi_\K)\in\k^*\otimes\R^3$, and let $\xi_\C\coloneqq\xi_\J+i\xi_\K\in\g^*$. 

\begin{proposition}\label{a6et6txr}
We have $\mu^{-1}(\xi)\s \mu_\C^{-1}(\xi_\C)\ass{\chi}$ and this inclusion descends to a homeomorphism $\mu^{-1}(\xi)/K\cong \mu_\C^{-1}(\xi_\C)\sll{\chi}G$. Moreover, $K$ acts freely on $\mu^{-1}(\xi)$ if and only if $G$ acts freely on $\mu_\C^{-1}(\xi_\C)\ass{\chi}$ and in that case the homeomorphism is an isomorphism of complex-symplectic manifolds.
\end{proposition}

\begin{proof}
By Theorem \ref{77svbtma}, we have $\mu_\I^{-1}(\xi_\I)\s M\ass{\chi}$ and this inclusion descends to a homeomorphism $\mu^{-1}_\I(\xi_\I)/K\to M\sll{\chi}G$. By the geometric criterion, we see that $M\ass{\chi}\cap\mu_\C^{-1}(\xi_\C)=\mu_\C^{-1}(\xi_\C)\ass{\chi}$. Thus, $\mu^{-1}(\xi)=\mu_\I^{-1}(\xi_\I)\cap\mu_\C^{-1}(\xi_\C)\s M\ass{\chi}\cap\mu_\C^{-1}(\xi_\C)=\mu_\C^{-1}(\xi_\C)\ass{\chi}$. Now, the map $\mu^{-1}(\xi)/K\to\mu_\C^{-1}(\xi_\C)\sll{\chi}G$ is a restriction of the homeomorphism $\mu_\I^{-1}(\xi_\I)/K\to M\sll{\chi}G$ so it suffices to show that it is surjective. Let $\pi:M\ass{\chi}\to M\sll{\chi}G$ be the quotient map and let $x\in\mu_\C^{-1}(\xi_\C)\sll{\chi}G$. Since $\mu_\C^{-1}(\xi_\C)$ is $G$-invariant and Zariski-closed in $M$, $\pi$ restricts to the GIT quotient $\pi:\mu_\C^{-1}(\xi_\C)\ass{\chi}\to \mu_\C^{-1}(\xi_\C)\sll{\chi}G$. Recall from GIT that there exists $p\in\pi^{-1}(x)$ which lies in the set $\mu_\C^{-1}(\xi_\C)\aps{\chi}$ of closed orbits in $\mu_\C^{-1}(\xi_\C)\ass{\chi}$. Also, we have $M\aps{\chi}=M\aps{(\mu_\I-\xi_\I)}$ which, by Theorem \ref{9ez7t5ob} (equation \eqref{siczshzr}), is equal to $G\cdot\mu_\I^{-1}(\xi_\I)$. Thus, $p\in\mu_\C^{-1}(\xi_\C)\aps{\chi}=M\aps{\chi}\cap \mu_\C^{-1}(\xi_\C)=(G\cdot\mu_\I^{-1}(\xi_\I))\cap\mu_\C^{-1}(\xi_\C)$ and hence there exists $g\in G$ such that $g\cdot p\in\mu^{-1}(\xi)$ (using that $\xi_\C$ is fixed by $G$). Then, the image of $g\cdot p$ in $\mu^{-1}(\xi)/K$ maps to $x$, so the map is surjective and hence a homeomorphism. That it is an isomorphism of complex-symplectic manifolds when the actions are free is a simple diagram chasing using the definition of reduced symplectic forms. 
\end{proof}

\section{Hyperk\"ahler quotients of $T^*G$}\label{phz5zi7d}

The goal of this section is to prove Proposition \ref{h6lxerau} and Proposition \ref{mtlkhb8t} in the introduction. By Proposition \ref{a6et6txr}, they imply Theorem \ref{p8uaulmf} about the hyperk\"ahler quotients of $T^*G$ by closed subgroups of $K\times K$ and hence also Theorem \ref{afdiktv2}. We use the notation introduced in \S\ref{m0mywa60}; in particular, $K$ is a compact connected Lie group, $G\coloneqq K_\C$, and $\M$ is the moduli space of solutions to the Nahm equations on $[0,1]$.

\subsection{The isomorphism $\M\cong T^*G$}\label{p1kintte}

Let us first recall, following Kronheimer \cite{kro88}, the isomorphism $\varphi:\M\to T^*G$ mentioned in Theorem \ref{knvp0kum}. There are, in fact, two intermediate isomorphisms $\M\to\N$ and $\N\to T^*G$, where $\N$ is the moduli space of solutions to the complex Nahm equation. That is, 
$$\N\coloneqq\{(\alpha,\beta):[0,1]\to\g\times\g:\dot{\beta}+[\alpha,\beta]=0\}/\G^0,$$
where $\G$ is the group of $C^2$ maps $[0,1]\to G$, $\G^0\coloneqq\{g\in\G:g(0)=g(1)=1\}$, and the action is
$$g\cdot(\alpha,\beta)=(g\alpha g^{-1}-\dot{g}g^{-1},g\beta g^{-1}).$$
The space $\N$ can be viewed as an infinite-dimensional complex-symplectic reduction and the map
$$\M\too\N,\quad A\mtoo(A_0+iA_1,A_2+iA_3)$$
is an isomorphism of complex-symplectic manifolds. For each solution $(\alpha,\beta)$ to the complex Nahm equation, there is a unique $g\in\G$ such that $g(0)=1$ and $g\cdot(\alpha,\beta)=(0,X)$ for some constant $X\in\g$. Indeed, simply solve the linear initial value problem $\dot{g}=g\alpha$, $g(0)=1$ and take $X=\beta(0)$. This defines a biholomorphism
\begin{equation}\label{qtrkjs27}
\N\too G\times\g,\quad (\alpha,\beta)\mtoo(g(1),\beta(0)),\quad (\dot{g}=g\alpha,g(0)=1).
\end{equation}
Now, identify $\g$ with $\g^*$ using the invariant bilinear form on $\g$ obtained by extending the $K$-invariant inner-product on $\k$. Then, $G\times\g\cong G\times\g^*\cong T^*G$, where the last isomorphism comes from right translations, i.e.\ $(g,\xi)\mto (dR_{g^{-1}})^*(\xi)$. Thus, we have a biholomorphism $\M\to\N\to T^*G$ which is, in fact, an isomorphism of complex-symplectic manifolds.

\subsection{$\M$ as an open subset of $K\times\k^3$}

The space $\M$ also has a convenient description as an open subset of $K\times\k^3$, as shown in Dancer--Swann \cite[Theorem 3]{dan96}. This finite-dimensional description of $\M$ will be useful for our analysis, so we review it here.

For each solution $A\in\A$ to the Nahm equations on $[0,1]$, there is a unique $k\in\cK$ such that $k(0)=1$ and $k\cdot A=(0,P_1,P_2,P_3)$ for some $P_i:[0,1]\to\k$. Indeed, we simply solve the linear initial value problem $\dot{k}=kA_0$, $k(0)=1$. Now, the $P_i$'s are solutions to the so-called \textit{reduced Nahm equation}
\begin{equation}
\begin{aligned}\label{69dhwv7x}
\dot{P}_1+[P_2,P_3] &= 0 \\
\dot{P}_2+[P_3,P_1] &= 0 \\
\dot{P}_3+[P_1,P_2] &= 0.
\end{aligned}
\end{equation}
The main point about this reduced form is that the number of equations is equal to the number of unknowns, so a solution $P=(P_1,P_2,P_3)$ is completely determined by its initial value $P(0)\in\k^3$. Thus, for each $X\in\k^3$, we let
\begin{equation}\label{q3jijnzf}
P^X\coloneqq\text{the unique solution to \eqref{69dhwv7x} with }P^X(0)=X,
\end{equation}
and
$$W\coloneqq\{X\in\k^3:P^X\text{ is defined at least on }[0,1]\}.$$
Then, we have a map
$$\M\too K\times W,\quad A\mtoo(k(1),A_1(0),A_2(0),A_3(0)),$$
where $k$ is the unique element of $\cK$ such that $\dot{k}=kA_0$ and $k(0)=1$. 

\begin{proposition}[Dancer--Swann \cite{dan96}]
The map $\M\to K\times W$ is a diffeomorphism and intertwines the $K\times K$-action on $\M$ with the action on $K\times W$ given by
\begin{equation*}
(k_1,k_2)\cdot(k,X)=(k_1kk_2^{-1},\Ad_{k_1}X).\tag*{\qed}
\end{equation*}
\end{proposition}

Under this diffeomorphism, the K\"ahler potential $F\coloneqq\frac{1}{2}(F_{12}+F_{13}):\M\to\R$ defined in \eqref{9v5jg218} becomes
\begin{equation}\label{s95ycna1}
\tilde{F}:K\times W\too\R,\quad \tilde{F}(k,X)=\frac{1}{4}\int_0^12\|P^X_1(t)\|^2+\|P^X_2(t)\|^2+\|P^X_3(t)\|^2\,dt.
\end{equation}
We will mainly work with $\tilde{F}$ rather than $F$.

\subsection{The diffeomorphism $K\times W\cong G\times\g$}

We want to give a more explicit description of the diffeomorphism
$$\psi:K\times W\too G\times\g$$
obtained by the composition $K\times W\to\M\to\N\to G\times\g$. Recall that the exponential map $\k\to K$ is surjective, so it suffices to describe $\psi$ for elements of the form $(e^Y,X)$.

\begin{lemma}\label{flghdtkj}
For all $Y\in\k$ and $X\in W\s\k^3$, we have
$$\psi(e^Y,X)=(g(1),X_2+iX_3),$$
where $g:[0,1]\to G$ is the unique solution to the linear ODE
\begin{equation}\label{81xuxt71}
\dot{g}(t)=g(t)(Y+i\Ad_{e^{-tY}}P^X_1(t))
\end{equation}
with $g(0)=1$ \textup{(}where $P^X_1$ is the first component of \eqref{q3jijnzf}\textup{)}. 
\end{lemma}

\begin{proof}
The diffeomorphism $K\times W\to \M$ is given by $(\gamma,X)\mto k_\gamma^{-1}\cdot(0,P^X)$ where $k_\gamma$ is any smooth map $[0,1]\to K$ such that $k_\gamma(0)=1$ and $k_\gamma(1)=\gamma$. In particular, $(e^Y,X)\mto e^{-tY}\cdot(0,P^X(t))=(Y,\Ad_{e^{-tY}}P^X(t))$ and the image of this in $\N$ is $(Y+i\Ad_{e^{-tY}}P^X_1(t),\Ad_{e^{-tY}}(P^X_2(t)+iP^X_3(t)))$. The definition of the isomorphism $\N\to G\times\g$ in \eqref{qtrkjs27} concludes the proof. 
\end{proof}

It will also be useful to know the derivative of that map at the origin:

\begin{lemma}\label{220ui22k}
Under the identifications $T_{(1,0)}(K\times W)=\k^4$ and $T_{(1,0)}(G\times\g)=\g^2$, we have $d\psi_{(1,0)}(X_0,X_1,X_2,X_3)=(X_0+iX_1,X_2+iX_3)$.
\end{lemma}

\begin{proof}
Let $Y=X_0$ and $X=(X_1,X_2,X_3)$ and denote by $g^X_Y$ the unique solution to \eqref{81xuxt71} with $g^X_Y(0)=1$. Then, it is easy to check that $g^{sX}_{sY}(t)=g^X_Y(st)$ for all $s\in\R$. Thus,
\begin{align*}
d\psi_{(1,0)}(Y,X) &= \frac{d}{ds}\Big|_{s=0}\psi(e^{sY},sX) = \frac{d}{ds}\Big|_{s=0}(g^{sX}_{sY}(1),sX_2+isX_3) \\
&= \frac{d}{ds}\Big|_{s=0}(g^{X}_{Y}(s),sX_2+isX_3) \\
&= (g^X_Y(s)(Y+i\Ad_{e^{-sY}}P^{X}_1(s)),X_2+iX_3)\Big|_{s=0} \\
&= (Y+iX_1,X_2+iX_3).\qedhere
\end{align*}
\end{proof}

\subsection{The moment map}

We now prove Proposition \ref{mtlkhb8t}, which identifies the first component
$$\mu_\I:\M\too\k^*\times\k^*,\quad\mu_\I(A)=(A_1(0), -A_1(1))$$
of the hyperk\"ahler moment map $\mu:\M\to(\k^*\times\k^*)\otimes\R^3$ described in Proposition \ref{rupl8lue} in terms of the K\"ahler potential $F:\M\to\R$. Namely, we show that $\mu_\I(A)(Z)=dF(\I Z^\#_A)$ for all $A\in \M$ and $Z\in\k\times\k$.

We know by Lemma \ref{dpa42xux} that $dF(\I Z^\#_A)$ defines a moment map. Since $\mu_\I(0)=(0,0)$ and moment maps are unique up to an additive constant, it suffices to show that $dF(\I Z^\#_0)=0$ for all $Z\in\k\times\k$. Equivalently, we show that $d\tilde{F}(\I Z^\#_0)=0$, where $\tilde{F}:K\times W\to\R$ is the pullback of $F$ given by \eqref{s95ycna1}. If $Z=(Z_1,Z_2)\in\k\times\k$ then on $T_{(1,0)}T^*G=\g\times\g$ we have
$$\I(Z_1,Z_2)^\#_{(1,0)}=\frac{d}{dt}\Big|_{t=0}(\exp(tiZ_1)\exp(-tiZ_2),0)=(i(Z_1-Z_2),0).$$
Thus, by Lemma \ref{220ui22k}, it suffices to show that $d\tilde{F}_{(1,0)}$ vanishes on $0\times\k\times 0\times 0\s\k^4$. Note that for $X\in\k^3$ and $s\in\R$ we have $P^{sX}(t)=sP^X(st)$. Hence, for all $X\in\k^3$,
\begin{align*}
d\tilde{F}_{(1,0)}(0,X_1,X_2,X_3) &= \frac{d}{ds}\Big|_{s=0} \tilde{F}(1,sX) \\
&= \frac{d}{ds}\Big|_{s=0}\frac{1}{4}\int_0^12\|P^{sX}_1(t)\|^2+\|P^{sX}_2(t)\|^2+\|P^{sX}_3(t)\|^2\,dt \\
&= \frac{d}{ds}\Big|_{s=0} \frac{s^2}{4}\int_0^12\|P^{X}_1(st)\|^2+\|P^{X}_2(st)\|^2+\|P^{X}_3(st)\|^2\,dt \\
&= 0.
\end{align*}
This concludes the proof of Proposition \ref{mtlkhb8t}.

\subsection{Growth rate of the potential}

We now prove Proposition \ref{h6lxerau}, which says that $\C[T^*G]\s o(e^f)$, where $f:T^*G\to\R$ is the K\"ahler potential corresponding to $F:\M\to\R$.

We view $T^*G$ as an affine variety in $\C^N$ for some $N>0$ and endow $\C^N$ with a norm $|\cdot|$. Then, Proposition \ref{h6lxerau} will be a consequence of the following estimate.

\begin{proposition}\label{2omnlkg4}
There exists $b,c>0$ and a compact set $B\s T^*G$ such that $|x|^2\le b e^{c\sqrt{f(x)}}$ for all $x\in T^*G\setminus B$.
\end{proposition}

\begin{proof}
To show this, let
$$\rho:W\too\R,\quad  \rho(X)= \frac{1}{4}\int_0^12\|P^X_1(t)\|^2+\|P^X_2(t)\|^2+\|P^X_3(t)\|^2\,dt$$
so that $\tilde{F}(k,X)=\rho(X)$ for all $(k,X)\in K\times W$. We can view the diffeomorphism $\psi:K\times W\to G\times\g$ as taking values in $\C^N$. Hence, the proposition can be reformulated as saying that there exists $b,c>0$ and a compact set $D\s W$ such that
\begin{equation}\label{p5sm9yjg}
|\psi(k, X)|^2<be^{c\sqrt{\rho(X)}}
\end{equation}
for all $(k,X)\in K\times(W\setminus D)$.

It was shown in the author's paper \cite{may17} that $\rho$ is a proper map; we will need this.

\begin{lemma}[{\cite[\S4.2]{may17}}]\label{lkkt47vt}
The map $\rho$ is proper and hence so is $f$.\qed
\end{lemma}

We will not reproduce the proof here, but the following intermediate step will be useful: 
\begin{lemma}\label{ugbcz4nk}
There exists a compact set $C\s W$ and a constant $\beta>0$ such that
$$\|X\|\le \beta \rho(X)$$
for all $X\in W\setminus C$ \textup{(}where $\|\cdot\|$ is the $K$-invariant inner-product on $\k$\textup{)}.
\end{lemma}

\begin{proof}
This follows from the proof of \cite[Lemma 4.5]{may17}.
\end{proof}

Let us now be more explicit about $\C^N$ and the choice of a norm on it. Throughout this proof, we view $G$ as a subgroup of $\SL(n,\C)$ for some $n>0$ and $K$ as a subgroup of $\SU(n)\s\SL(n,\C)$. Then, $G\times\g$ can be viewed as an affine variety in $\C^{2n^2}=\gl(n,\C)\times\gl(n,\C)$ by considering the embedding $G\times\g\s \SL(n,\C)\times\sl(n,\C)\s\gl(n,\C)\times\gl(n,\C)$. Take the standard norm ${|\cdot|}$ on $\gl(n,\C)$, i.e.\ $|X|^2=\sum_{ij}|X_{ij}|^2$ and the product norm on $\gl(n,\C)\times\gl(n,\C)$. This is the norm we use to estimate $\psi$. First, we have the following basic fact.

\begin{lemma}\label{oqh7raiw}
$|XY|\le n^2|X||Y|$ for all $X,Y\in\gl(n,\C)$.
\end{lemma}

\begin{proof}
Indeed,
\begin{align*}
|XY|^2 &= \sum_{i,j}\left(\sum_k X_{ik}Y_{kj}\right)^2 \le \sum_{i,j}\left(\sum_k |X||Y|\right)^2 = n^4|X|^2|Y|^2.\qedhere
\end{align*}
\end{proof}

We will also need to following standard result from the theory of ODEs.

\begin{lemma}[Gr\"onwall \cite{gro19}]
Let $u:[0,t_0]\to\R$ be differentiable. If there is a continuous function $\beta:[0,t_0]\to\R$ such that $\dot{u}(t)\le\beta(t)u(t)$ for all $t\in[0,t_0]$, then
$$u(t)\le u(0)e^{\int_0^t\beta(s)ds}$$
for all $t\in[0,t_0]$.\qed
\end{lemma}

Finally, we observe the following fact.

\begin{lemma}\label{lyhplhnf}
Let $K$ be a compact connected Lie group. Then, there exists a compact set $B\s\k$ such that $K=\exp(B)$.
\end{lemma}

\begin{proof}
Let $T\s K$ be a maximal torus and let $\t\s\k$ be its Lie algebra. Then, the restriction of $\exp$ to $\t$ is of the form
$$\R^m\too (S^1)^m,\quad (\theta_1,\ldots,\theta_m)\mtoo (e^{i\theta_1},\ldots,e^{i\theta_m}).$$
Thus, $B' \coloneqq [0,2\pi]^m$ is a compact subset of $\t$ such that $\exp(B')=T$. Now, let $B\coloneqq K\cdot B'=\{\Ad_kX:k\in K,X\in B'\}$. Then, $B$ is compact since $K$ and $B'$ are compact. Since every element of $K$ is conjugate to an element of $T$ (see e.g.\ \cite[Theorem 4.36]{kna02}), we have $\exp(B)=K$.
\end{proof}
With these preliminaries, we can now prove the proposition. Let $B$ be as in Lemma \ref{lyhplhnf} and let $r>0$ be such that $B$ is contained in the ball of radius $r$ centred at $0$ in $\k$ in the norm $|\cdot|$. The restriction of $|\cdot|$ to $\k$ might not be the same as the norm $\|\cdot\|$ induced by the $K$-invariant inner-product, but since $\k$ is finite-dimensional, there exists $c_0,c_1>0$ such that $c_0\|X\|\le|X|\le c_1\|X\|$ for all $X\in\k$. Let
$$b\coloneqq 2e^{2rn^2}\quad\text{and}\quad c\coloneqq2\sqrt{2}n^2c_1,$$
where the integer $n$ is the same as the one used for the embedding $G\s\SL(n,\C)$. By Lemma \ref{ugbcz4nk}, there exists a compact set $C \s W$ and $\alpha>0$ such that $|X|^2\le\alpha\rho(X)^2$ for all $X\in W\setminus C$. Note that the set
$$\{t\in[0,\infty): \frac{b}{2}e^{c\sqrt{t}}\le \alpha t^2\}$$
is contained in an interval $[0,\beta]$ for some $\beta>0$, so
\begin{align*}
D &\coloneqq \{X\in W : \frac{b}{2}e^{c\sqrt{\rho(X)}} \le |X|^2\} \\
&\s C\cup\{X\in W: \frac{b}{2}e^{c\sqrt{\rho(X)}}\le \alpha\rho(X)^2\} \\
&\s C\cup\rho^{-1}([0,\beta]).
\end{align*}
Since $\rho$ is proper (Lemma \ref{lkkt47vt}), $\rho^{-1}([0,\beta])$ is compact in $W$, and hence $D$ is also compact. We claim that \eqref{p5sm9yjg} holds with those $b$, $c$ and $D$.

Let $(k,X)\in K\times W$ and write $k=\exp(Y)$ for some $Y\in B$ (so $|Y|\leq r$). By Lemma \ref{flghdtkj}, we have
$$\psi(k,X)=(g(1),X_2+iX_3),$$
where $g:[0,1]\to G$ is the unique solution to the linear ODE
$$
\dot{g}(t)=g(t)(Y+i\Ad_{e^{-tY}}P^X_1(t))
$$
with $g(0)=1$. Let $u:[0,1]\to \R$ be defined by $u(t)=|g(t)|^2$. Then,
$$\dot{u}(t) = 2\ip{g(t),\dot{g}(t)}=2\ip{g(t),g(t)(Y+i\Ad_{e^{-tY}}P^X_1(t))}.$$
By the Cauchy-Schwarz inequality and Lemma \ref{oqh7raiw},
\begin{align*}
\dot{u}(t) &\le 2|g(t)||g(t)(Y+i\Ad_{e^{-tY}}P^X_1(t))| \\
&\le 2n^2 |g(t)|^2|Y+i\Ad_{e^{-tY}}P^X_1(t)| \\
&\le 2n^2(r+|P_1^X(t)|)u(t).
\end{align*}
By Gr\"onwall's Lemma,
\begin{align*}
|g(1)|^2 = u(1) &\le \exp\left(2n^2\int_0^1(r+|P_1^X(s)|)ds\right) = \exp\left(2rn^2+2n^2\int_0^1|P_1^X(s)|ds\right).
\end{align*}
By H\"older's inequality, 
\begin{align*}
\int_0^1|P_1^X(s)|ds &\le \left(\int_0^1|P_1^X(s)|^2ds\right)^{1/2}\left(\int_0^1ds\right)^{1/2} \\
&\le \left(\frac{c_1^2}{2}\int_0^1(2\|P_1^X(s)\|^2+\|P_2^X(s)\|^2+\|P_3^X(s)\|^2)ds\right)^{1/2} \\
&= \sqrt{2}c_1\sqrt{\rho(X)},
\end{align*}
so
$$|g(1)|^2\le \exp(2rn^2+2\sqrt{2}n^2c_1\sqrt{\rho(X)})=\frac{b}{2}e^{c\sqrt{\rho(X)}}.$$
Therefore,
$$|\psi(k,X)|^2=|g(1)|^2+|X_2|^2+|X_3|^2\le |X|^2+ \frac{b}{2} e^{c\sqrt{\rho(X)}}.$$
Then, when $X\in W\setminus D$, we have
\begin{equation*}
|\psi(k,X)|^2\le |X|^2+\frac{b}{2}e^{c\sqrt{\rho(X)}}< be^{c\sqrt{\rho(X)}},
\end{equation*}
where the last inequality follows from the definition of $D$.
\end{proof}

\begin{lemma}\label{xfap1j1v}
The function $e^{f}$ dominates $\alpha e^{\beta{\sqrt{f}}}$ for all $\alpha,\beta>0$.
\end{lemma}

\begin{proof}
This follows directly from properness of $f$. Indeed, for all $\gamma>0$, the set
$$B\coloneqq\{t\in[0,\infty):\alpha e^{\beta\sqrt{t}}\ge \gamma e^t \}$$
is compact, so $C\coloneqq f^{-1}(B)$ is also compact. If $x\notin C$ we have $f(x)\notin B$ and hence $\alpha e^{\beta\sqrt{f(x)}}<\gamma e^{f(x)}$.
\end{proof}

\begin{proof}[Proof of Proposition \ref{h6lxerau}]
Let $\gamma>0$ and $u\in\C[T^*G]$. We want to show that there exists a compact set $C\s T^*G$ such that $|u(x)|\le\gamma e^{f(x)}$ for all $x\in T^*G\setminus C$. We view $T^*G$ as an affine variety in $\C^N$ for some $N>0$ and write $u(x)=\sum a_{i_1\ldots i_N}x_1^{i_1}\cdots x_N^{i_N}$. Then, $|u(x)|\le\sum_{k=0}^n a_k|x|^k$ for some $a_k>0$ and $n\ge 0$. By Proposition \ref{2omnlkg4}, there exists a compact set $B\s T^*G$ and $b,c>0$ such that $|x|^2\le be^{c\sqrt{f(x)}}$ for all $x\in T^*G\setminus B$. Hence, for all $x\in T^*G\setminus B$ we have $|u(x)|\le\sum_{k=0}^n a_kb^{\frac{k}{2}} e^{\frac{ck}{2}\sqrt{f(x)}}$. By Lemma \ref{xfap1j1v}, for all $k\in\{0,\ldots,n\}$ there exists a compact set $C_k$ such that if $x\notin C_k$ then $a_kb^{\frac{k}{2}}e^{\frac{ck}{2}\sqrt{f(x)}}\le\frac{\gamma}{n+1}e^{f(x)}$. Let $C=B\cup C_0\cup\cdots\cup C_n$. Then, for all $x\notin C$, we have $|u(x)|\le\sum_{k=0}^n\frac{\gamma}{n+1}e^{f(x)}=\gamma e^{f(x)}$.
\end{proof}

\end{document}